\documentclass[11pt,a4paper,reqno]{amsart}

\calclayout

\usepackage[utf8]{inputenc}
\usepackage[T1]{fontenc}

\usepackage{appendix}

\usepackage{amsmath, amsthm, amsfonts, amssymb}
\usepackage{mathtools}
\usepackage{bbm} 

\numberwithin{equation}{section} 
\numberwithin{table}{section} 
\numberwithin{figure}{section} 

\newdimen\oldparindent
\usepackage{parskip} 
\makeatletter 
\def\thm@space@setup{ \thm@preskip=\parskip \thm@postskip=0pt }
\makeatother

\usepackage[shortlabels]{enumitem} 
\setlist[enumerate, 1]{label={(\roman*)}} 

\usepackage[backend=bibtex, sorting=nyt, giveninits=true]{biblatex}

\renewbibmacro{in:}{%
  \ifboolexpr{%
     test {\ifentrytype{article}}%
     or
     test {\ifentrytype{inproceedings}}%
  }{}{\printtext{\bibstring{in}\intitlepunct}}%
}

\usepackage{hyperref} 
\hypersetup{
  colorlinks   = true, 
  urlcolor     = blue, 
  linkcolor    = blue, 
  citecolor    = red   
}
\usepackage{url}
\newtheorem{theorem}{Theorem}[section]
\newtheorem{definition}[theorem]{Definition}
\newtheorem{proposition}[theorem]{Proposition}
\newtheorem{corollary}[theorem]{Corollary}
\newtheorem{lemma}[theorem]{Lemma}

\newtheorem*{definition*}{Definition}
\newtheorem*{theorem*}{Theorem}
\newtheorem*{proposition*}{Proposition}
\newtheorem*{corollary*}{Corollary}
\newtheorem*{lemma*}{Lemma} 
\newtheorem*{remark*}{Remark}
\newtheorem*{notation*}{Notation}
\newtheorem*{conjecture*}{Conjecture}

\DeclareNameAlias{default}{family-given}
\AtBeginDocument{%
  \renewcommand*\mkbibnamelast[1]{\textsc{#1}}}
  
\let\TC\textcite
\renewcommand\textcite[1]{{\def\mkbibnamelast##1{##1}\TC{#1}}}
\renewbibmacro{in:}{}            

\def\eps{\varepsilon}

\title{Cutoff for the logistic SIS epidemic model with self-infection}

\author{Roxanne He}
\address{School of Mathematics and Statistics, The University of Melbourne, Parkville, VIC, 3010, Australia}
\email{hey4@student.unimelb.edu.au}

\author{Malwina Luczak}
\address{Department of Mathematics, University of Manchester}
\email{malwina.luczak@manchester.ac.uk}
\thanks{The research of Malwina Luczak is currently supported by the Leverhulme Trust, grant reference LIP-2022-005, and was previously supported by an 
Australian Research Council Future Fellowship (FT170104009).}

\author{Nathan Ross}
\address{School of Mathematics and Statistics, The University of Melbourne, Parkville, VIC, 3010, Australia}
\email{nathan.ross@unimelb.edu.au}

\date{\today}

\begin{document}

\begin{abstract}
We study a variant of the classical Markovian logistic SIS epidemic model on a complete graph, which has the additional feature that healthy individuals can become infected without contacting an infected member of the population. This additional ``self-infection'' is used to model situations where there is an unknown source of infection or an external disease reservoir, such as an animal carrier population. In contrast to the classical logistic SIS epidemic model, the version with self-infection has a non-degenerate stationary distribution, and we derive precise asymptotics for the time to converge to stationarity (mixing time) as the population size becomes large. It turns out that the chain exhibits the cutoff phenomenon, which is a sharp transition in time from one to zero of the total variation distance to stationarity. We obtain the exact leading constant for the cutoff time, and show the window size is constant (optimal) order.
While this result is interesting in its own right, an additional contribution of our work is that the proof illustrates a recently formalised methodology of Barbour, Brightwell and Luczak~\cite{barbour2022long}, which can be used to show cutoff via a combination of concentration of measure inequalities for the trajectory of the chain, and coupling techniques. 
\end{abstract}



\maketitle

\section{Introduction}
Mathematical models of disease spread are widely used in epidemiology, to assist in preventing and managing epidemic outbreaks.   
One well-studied epidemic model is the logistic Susceptible-Infected-Susceptible model (SIS model) on the complete graph with $N$ vertices, which is a continuous time birth-and-death  Markov chain $X^*_N=(X^*_N(t))_{t\geq0}$ with state space $\{0, 1, \ldots, N\}$, and transitions given by
\begin{align*}
    &x^*\mapsto x^*+1\text{ at rate }\lambda x^*(1-x^*/N),\\
    &x^*\mapsto x^*-1\text{ at rate }\mu x^*,
\end{align*}
where $x^*\in\{0,1, \ldots, N\}$ and $\lambda, \mu>0$. The model describes an epidemic spreading in a closed population with $N$ individuals,
with the number of infected people at time $t$ represented by $X^*_N(t)$. To explain the dynamics, each infected person encounters a random member of the population at rate $\lambda$, and if the other individual is currently susceptible to infection, that individual then becomes infected. Each infected individual recovers at rate $\mu$, and once they are recovered, they immediately become susceptible again. If, at any time, the number of infected people becomes zero, we say that the epidemic outbreak is extinct. 
This model was first formulated and studied by Feller~\cite{feller_1939} in the 1930s, but was not heavily
studied until the 1970s, when it was popularised by Weiss and Dishon~\cite{weiss_dishon_1971}.

We are interested in asymptotics as $N \to \infty$, with $\lambda$, $\mu$  fixed positive constants.  We use the phrase ``with high probability'' to mean ``with probability tending to $1$ as $N \to \infty$''. 
A key quantity in the SIS model is the so-called basic reproduction ratio $\mathcal{R}_0:=\lambda/\mu$. The epidemic is called supercritical if $\mathcal{R}_0>1$, and subcritical if $\mathcal{R}_0<1$. It is well-known that for the supercritical SIS model, if there are a large number of infected individuals at time~0, then with high probability, $X_N^*(t)/N$ heads rapidly towards the stable fixed point of a corresponding deterministic model given by a (logistic) differential equation, then spends most of its time before extinction in the neighbourhood of that fixed point. The time to extinction is exponential in $N$ (in mean and with high probability) as $N\to\infty$; see, for example, N\aa sell~\cite{naasell1996quasi}. For more detailed aspects of the supercritical logistic SIS model such as quasi-stationarity, see, for example, Andersson and Djehiche~\cite{andersson_djehiche_1998}, Barbour~\cite{barbour1976quasi} and N\aa sell~\cite{nasell_1999}. 

In the subcritical case where $\lambda<\mu$, Doering et al.~\cite{doering_sargsyan_sander_2005} showed in 2005 that, if the initial state $X_N^*(0)$ is of order $N$, then the expected extinction time is $\frac{1}{\mu-\lambda} (\log N + O(1))$. 

Brightwell et al.~\cite{brightwell2018extinction} obtained a formula for the asymptotic distribution of the extinction time in the subcritical case, and, in particular, proved that, if the starting state is of order $N$, then for any constant $c>0$ large enough, the probability of extinction by time $\frac{1}{\mu-\lambda}\log(N)-c$ is nearly zero, while by time $\frac{1}{\mu-\lambda}\log(N)+c$, it is nearly one. To be precise, an asymptotic formula for the extinction time, where the randomness has a Gumbel distribution, is given in~\cite{brightwell2018extinction} in the case where $\mu > \lambda$ and both are fixed constants. The authors also consider the case where $\lambda, \mu$ are allowed to vary with $N$, and $\mu (N) - \lambda (N) \to 0^+$ suitably slowly (the ``barely subcritical'' case) and show that the same formula for the extinction time holds in this regime.

Recent work of Foxall \cite{foxall2021extinction} 
summarises and derives further results for the asymptotic behaviour of the extinction time under different asymptotic regimes for the initial infection size and reproduction ratio (also allowing $\lambda$ and $\mu$ to vary with $N$).

The model studied in this paper is a variant of the logistic SIS epidemic model known as the logistic SIS epidemic model with self-infection ($\varepsilon$-SIS model), which is a continuous-time birth-and-death  Markov chain $X_N=(X_N(t))_{t\geq0}$ with state space $\{0, 1, \ldots, N\}$, and transitions given by
\begin{align*}
    &x\mapsto x+1\text{ at rate }\lambda x(1-x/N)+\varepsilon(N-x),\\
    &x\mapsto x-1\text{ at rate }\mu x,
\end{align*}
where $x\in\{0,1,\ldots,N\}$ and $\lambda, \mu, \varepsilon>0$. A literature review for this lesser studied model is given later in the present section.

Compared to the standard SIS model, the $\varepsilon$-SIS model has an additional source of infection, which can be thought of as an infection from an unknown source or external disease reservoir, such as an animal carrier of the disease. Thus, any susceptible person can be infected (at rate~$\eps$)  without having contact with any infected member of the population. 
Excluding the case $\eps=0$ (done throughout this paper), which is the classical well-studied SIS model, 
the $\eps$-SIS model admits a non-degenerate stationary distribution, and so the relevant quantity of interest is the time to stationarity, which is analogous to the extinction time in the classical SIS model. We are interested in asymptotics as the population size $N$ tends to infinity. Assuming the parameters of the model, $\lambda$, $\mu$ and $\varepsilon$ are absolute constants not depending on $N$, we show that the sequence of $\eps$-SIS Markov chains $(X_N)_{N\geq1}$ exhibits a cutoff, as per the following definition.

\begin{definition}\label{cutoffphenomenon}
    The sequence of Markov chains $\{X_N(t):t \ge 0\}_{N\geq1}$ with finite state space $E_N$, transition function $P_N^t(x, \,\cdot\,)$, and stationary distribution $\pi_N$ is said to exhibit a cutoff at time $t_N$ with window size $\omega_N$, if $\omega_N=o(t_N)$ and 
\begin{equation*}
    \lim_{s\to\infty}\liminf_{N\to\infty} \rho_N(t_N-s\omega_N)=1, \quad \lim_{s\to\infty}\limsup_{N\to\infty} \rho_N(t_N+s\omega_N)=0,
\end{equation*}
where 
\begin{equation*}
    \rho_N(t) := \max_{x\in E_N} \lVert P_N^t(x,\,\cdot\,)-\pi_N \rVert_{\text{TV}},
\end{equation*}
is the worst-case total variation distance between the distribution of $X_N(t)$ and the chain's stationary distribution $\pi_N$. 
\end{definition}
Informally, cutoff means that, for any starting state, the distribution of the chain goes from being about as far as possible in total variation distance from the stationary distribution to about as close as possible, over a window of time of length $\omega_N$ centred around time $t_N$.

The cutoff phenomenon for Markov chains was first identified in the 1980s for random transpositions on the symmetric group by Diaconis and Shahshahani \cite{diaconis1981generating}, and the name was coined by Aldous and Diaconis \cite{aldous1986shuffling}. Establishing  cutoff for specific models of Markov chains is to this day an active area of research and in general a difficult problem. 
For the model in question, we have the following main result of the paper, which gives the precise cutoff time of the chain, with an optimal window.
\begin{theorem}\label{cutoff}
The sequence of $\eps$-SIS Markov chains $(X_N)_{N\geq1}$ has a cutoff at $t_N:=\frac{1}{2J}\log N$ with constant window size, where $J:= \sqrt{(\lambda-\mu-\varepsilon)^2+4\lambda\varepsilon}$.
\end{theorem}

Before going further, we highlight a few points about the theorem. First,  Theorem~\ref{cutoff} implies that  for a very small $\varepsilon$, the cutoff time of the $\varepsilon$-SIS model is about $\frac{1}{2(\mu-\lambda)}\log(N)$. This is about twice as fast as the subcritical classical SIS model, which has a cutoff at $\frac{1}{\mu-\lambda}\log(N)$.  This finding does not appear to be at all a priori obvious: we will give an intuitive explanation for it after describing the proof in more detail at the end of the introduction.

Second, 
general conditions that guarantee a sequence of birth-and-death chains exhibits cutoff 
are given in Ding et al.~\cite{ding2010total}. 
They show that, if the sequence satisfies the so-called \emph{product condition}, that the \emph{relaxation times}
$t_{\rm{REL}}(N)$--defined to be one over the spectral gap of the chain--is of smaller order as $N$ goes to infinity than the \emph{mixing time} $t_{\rm{MIX}}(N)$--defined to be the first time the total variation distance to stationarity is less than $1/4$--then the chain exhibits cutoff with window of order at most $\sqrt{t_{\rm{REL}}(N)t_{\rm{MIX}}(N)}$. 
While it is likely possible in our setting  to verify the product condition, or related conditions, see for example the follow-up works \cite{Chen2013b, Chen2014a, Chen2015a},
we would still need to derive bounds on the mixing time, and the window size $\sqrt{t_{\rm{REL}}(N)t_{\rm{MIX}}(N)}$ is at best of order $\sqrt{\log(N)}$, because the relaxation time is at least one, and the mixing time is of order $\log(N)$.
Thus our Theorem~\ref{cutoff} is sharper than what would be obtained from these general results, by giving the constant for the cutoff time, and showing the window is of constant order.
Ding et al.~\cite{ding2010total} 
also show the window size given in their results is unimprovable for general birth-and-death chains, so our result is also interesting as a non-trivial example where the window is smaller than predicted by the general result.

We also mention in this vein of research the work of 
Basu et al.~\cite{basu2014characterization}, which gives analogous conditions for cutoff of reversible Markov chains on finite trees, and semi birth-and-death chains, and the very recent work of Salez~\cite{salez2023cutoff}, which gives a ``product-like'' condition for cutoff assuming that the Markov chain has ``non-negative curvature''. 
These works also give a cutoff window 
 of order at most $\sqrt{t_{\rm{MIX}}(N)}\cdot t_{\rm{REL}}(N)^\alpha$ for an $\alpha>0$, which,  as in the birth-and-death case, may not be sharp. Thus, in many applications precise information about the exact cutoff time and window size needs to be obtained by ``bare hands'' methods.

The technique we use for proving Theorem~\ref{cutoff} is adapted from Barbour et al.~\cite{barbour2022long}, where the authors developed a general approach to establishing the cutoff phenomenon for suitable chains by using couplings and concentration of measure inequalities; see also that paper for a comprehensive literature review of the cutoff phenomenon.  The approach has been used in other places such as Lopes and Luczak \cite{lopes2020extinction}, who derived a formula for the asymptotic distribution of the extinction time for the weaker of two competing SIS epidemics (while that result is not an example of the cutoff phenomenon, it has a similar flavour); and Eskenazis and Nestoridi \cite{eskenazis2020cutoff}, who established the cutoff phenomenon for the Bernoulli--Laplace urn model with $o(N)$ swaps. Thus, an additional contribution of our work is to provide another example of this approach, which we hope will serve to standardise it as a tool
for establishing cutoff.

A final contribution of this article is to bring the $\eps$-SIS model to the attention of researchers in applied probability, and so we next discuss the existing literature. Afterwards, the end of the introduction gives an overview of the proof of Theorem~\ref{cutoff}, and simultaneously the organisation of the paper.

\noindent{\bf Related literature.}
The $\varepsilon$-SIS model has appeared sporadically in the applied literature. The theoretical primer (aimed at applied researchers) of Keeling and Ross \cite{keeling2008methods} includes the additional~$\eps$ term in their definition of the (vanilla) SIS model. Stone et al. \cite{Stone2008} apply the $\eps$-SIS model to head lice epidemic data, after deriving closed-form analytic formulas for some stationary quantities, and Hill et al. \cite{Hill2010} use it to model obesity spread. Nieddu et al. \cite{nieddu2022characterizing}  derive analytic quantities of the model to understand the parameter space where the disease is endemic. 

The most systematic study of the model appears in a series of papers by
Van Mieghem, alone and with and a number of co-authors~
\cite{Achterberg2022, mieghem_2020, van2012epidemics, mieghem_wang_2020}.  
Van Mieghem~\cite{mieghem_2020} uses an exact expression for the mean prevalence (i.e., the expectation of $X_N$) in equilibrium to identify a sharp phase transition in this mean as $\lambda/\mu$ varies, for fixed $N$ and a fixed very small positive value of $\varepsilon$. A follow-up paper of Van Mieghem and Wang~\cite{mieghem_wang_2020} studies time-dependent behaviour in this regime. 
We remark here that known results about the supercritical classical logistic SIS model ($\lambda/\mu > 1$) help to explain this phase transition phenomenon.  
When $\lambda/\mu > 1$ and $\varepsilon$ is much smaller compared to $1/N$, the behaviour of the $\varepsilon$-SIS model is very similar to that of the usual logistic SIS model in the supercritical regime, except that, after reaching state $0$, the process restarts after the next self-infection, so that the mean waiting time in state $0$ is $(N\varepsilon)^{-1}$.  It is known (see, for instance, \cite{andersson_djehiche_1998}) that the duration of the supercritical logistic epidemic is asymptotically exponential with a mean whose leading term is $e^{\gamma N}$, with $\gamma = \log (\lambda / \mu) - 1 + \mu/\lambda$. If $(N\epsilon)^{-1}$ is much smaller than $e^{\gamma N}$, then the mean prevalence is of order $N$, as the $\varepsilon$-SIS model will be spending most of its time near the stable fixed point of the logistic equation at $1-\mu/\lambda$. On the other hand, if $(N\epsilon)^{-1}$ is much larger than $e^{\gamma N}$, then the mean prevalence is near zero, since then most of the time the epidemic will be extinct, waiting to restart. 

While papers~\cite{mieghem_2020,mieghem_wang_2020} assume that the population size $N$ is fixed, the regime considered there is qualitatively very similar to the case where $N$ is very large, $\varepsilon$ is exponentially small in $N$, and $\lambda, \mu$ fixed constants, and therefore very different to the regime considered in the present paper.
 

%


\noindent{\bf Proof and paper overview.}
 In Section~\ref{detsismodel} we provide a brief discussion of the deterministic $\varepsilon$-SIS model, represented by the differential equation given in~\eqref{epssis}. The solution to this differential equation provides an approximation for the scaled process $(X_N(t)/N)_{t\geq0}$ over an appropriate time scale, for large $N$.

 In Section~\ref{priorconcen}, we show in Lemma~\ref{following} that, for a (deterministic) period of time that is stretched exponential in $N$, with high probability, the scaled process $(X_N(t)/N)$ is at most $O(N^{-\frac{1-h}{2}})$, for some fixed $h\in(0,1),$ away from the solution of equation~\eqref{epssis}.  In addition, we show in Lemma~\ref{mean} that for large enough times $t$, the mean of $X_N(t)/N$ is within $O(N^{-1/2})$ of the stable fixed point $x^{\star}$  of~\eqref{epssis}, given explicitly at~\eqref{fixedpoint}.

In Section~\ref{coupling}, we obtain an upper bound on the mixing time of the $\varepsilon$-SIS model.
We show that by time $t_N+\xi$, where $\xi$ is a constant that is independent of any parameters of the model and the population size $N$, the distribution of the chain $X_N(t)$ is sufficiently close to its stationary distribution $\pi_N$. This is done by using the Markov chain coupling method, where the coupling used consists of running two copies of the chain independently, over three phases. During the \emph{burn-in} phase, the concentration results of Section~\ref{priorconcen} imply that by time $\frac{1-h}{2J}\log(N)$, both copies are within   $O(N^{\frac{1+h}{2}})$ of $x^{\star}N$ with high probability; see Corollary~\ref{burn}.  Given that this event has occurred, during the \emph{intermediate} phase, the copies come within $O(\sqrt{N})$ of each other after an additional $\frac{h}{2J}\log(N)+\xi/2$ time with high probability; see Lemma~\ref{med}. In the \emph{final} phase, given the copies are within $\sqrt{N}$ of each other, 
their distance may be compared to an unbiased random walk, which hits $0$ from a state of order $\sqrt{N}$ within a constant time, given that the transition rate is of order $N$; see Lemma~\ref{final}.

In Section~\ref{lowerSIS}, we obtain a lower bound on the mixing time. We first show an improved version of the concentration inequality given in Section~\ref{priorconcen}, which only holds if the starting state of the chain $X_N(t)$ is in a ``good set''. We then break the proof of the lower bound into two parts. First, we combine the rapidly mixing result from the previous section and the improved concentration inequality to show the stationary distribution $\pi_N$ is concentrated within a ball of order $\sqrt{N}$ around $x^{\star}N$. Then we use concentration around the deterministic solution to show that most of the mass of $X_N(t_N-\xi)$ is located outside that ball, and thus by the time $t_N-\xi$ the distribution of the chain cannot be close to stationary.  

Finally, as mentioned already, for very small $\eps$ the $\varepsilon$-SIS model mixes twice as fast as the classical subcritical SIS model. To see why this is the case, for the $\eps$-SIS model, the (non-normalised) process arrives within distance $\sqrt{N}$ of the fixed point $x^{\star}N$, where most of the mass of the stationary distribution lies, after time of order about $\frac{1}{2(\mu-\lambda)}\log(N)$. After that, the stochastic model can be shown to do better than the deterministic one, as we can compare its behaviour to that of an unbiased random walk in continuous time taking steps at rate of order $N$, since the transition rates remain of order $N$ throughout. Such a random walk will hit $0$ in constant time with high probability. 

For the classical subcritical SIS model, the process also arrives within $\sqrt{N}$ of $0$ after about $\frac{1}{2(\mu-\lambda)}\log(N)$ time. However, at that point events happen at rate of order only $\sqrt{N}$ and decrease as we get closer to $0$, so any comparison with an unbiased random walk would be useless. Instead, we need to keep taking advantage of the negative drift right until the end, first following the differential equation closely, and then approximating by a linear birth-and-death process.

\section{Deterministic version of the \texorpdfstring{$\varepsilon$-}-SIS model}\label{detsismodel}
In this section, we consider the deterministic version of the logistic SIS model with self-infection. The model represents a spreading epidemic governed by the autonomous ordinary differential equation (ODE) 
\begin{equation}\label{epssis}
    \frac{dx}{dt}=\lambda x(1-x)+\varepsilon(1-x)-\mu x, \quad t\geq0,
\end{equation}
where $\lambda, \mu, \varepsilon>0$. The proportion of infected people at time $t$ is modeled by $x(t)$. When $\varepsilon=0$, we recover the classical deterministic logistic SIS model. 

We solve the equation $f(x)=\lambda x(1-x)+\varepsilon(1-x)-\mu x=0$ to identify the \textit{fixed points}. There are two solutions to the equation, however, since only non-negative solutions have a biological meaning, we require $ x(t)\geq 0$ for all $t\geq0$ in \eqref{epssis}. The only fixed point that falls in this range is  
\begin{equation}\label{fixedpoint}
    x^{\star}=\frac{1}{2\lambda}\left((\lambda-\mu-\varepsilon)+\sqrt{(\lambda-\mu-\varepsilon)^2+4\lambda\varepsilon}\right):=\frac{\lambda-\mu-\varepsilon+J}{2\lambda},
\end{equation}
where we recall from the statement of Theorem~\ref{cutoff} that $J= \sqrt{(\lambda-\mu-\varepsilon)^2+4\lambda\varepsilon}$.
Since $x(t)$ models the proportion of infected people, $x^{\star}=(\lambda-\mu-\varepsilon+J)/2\lambda\leq1$, as one would expect. When $\varepsilon=0$, the solution simplifies to $\frac{\lambda-\mu}{\lambda}$, if $\lambda>\mu$, whereas if $\lambda<\mu$ the solution degenerates to zero. The other solution is 
\begin{equation*}
    x_1^{\star}=\frac{1}{2\lambda}\left((\lambda-\mu-\varepsilon)-\sqrt{(\lambda-\mu-\varepsilon)^2+4\lambda\varepsilon}\right):=\frac{\lambda-\mu-\varepsilon-J}{2\lambda}, 
\end{equation*}
which is non-positive and degenerates to zero when $\varepsilon=0$, provided $\lambda>\mu$. The following proposition is easily verified by simple calculations and standard results, for example, Theorem 2.4.2 in \cite{boyce2004elementary}.
\begin{proposition}\label{detsol}
    The differential  equation \eqref{epssis} subject to the initial condition $x(0)=\alpha\in[0,1]$ has an explicit solution
    \begin{equation}\label{parsolepssis}
    x(t)=x^{\star}+\frac{\frac{J}{\lambda}(\alpha-x^{\star})}{(\alpha-x_1^{\star})e^{tJ}-(\alpha-x^{\star})}, \quad \text{where} \quad J=\sqrt{(\lambda-\mu-\varepsilon)^2+4\lambda\varepsilon}.
\end{equation}
\end{proposition}
It is easy to see from \eqref{parsolepssis} that, as $t\to\infty$, the solution monotonically approaches $x^{\star}$ from below when $0\leq \alpha<x^{\star}$, whereas it monotonically approaches $x^{\star}$ from above if $x^{\star}<\alpha\leq 1$. This can also be seen from the phase diagram. The graph of $f(x)$ versus $x$ is a downward parabola, so $dx/dt>0$ when $0\leq x<x^{\star}$, and $dx/dt<0$, if $x^{\star}<x\leq1$. This implies that regardless of the starting point, any solution of \eqref{epssis} approaches $x^{\star}$ monotonically as $t\to\infty$, so the fixed point $x^{\star}$ is 
\textit{globally attractive}, over the set $[0,1]$. 

\section{Concentration around the deterministic process}\label{priorconcen}
For the rest of this paper, we use $x(t)$ to denote the solution to the governing equation \eqref{epssis} with initial state $x(0)=X_N(0)/N$. There are two main results of this section. We first show that the scaled process $X_N(t)/N$ closely follows the solution $x(t)$ defined in \eqref{parsolepssis}. To be precise, we prove the following result.
\begin{lemma}\label{following}
There exist positive constants $C_1, C_2$ depending only on  $\lambda, \mu$ and $\varepsilon$ such that for all $h\in(0,1)$, the following relation holds for all $N$ large enough (depending on $\lambda, \mu, \varepsilon$ and $h$; see~\eqref{eq:11}),
\begin{equation}\label{lowrangecon}
    \mathbb{P}\left(\sup_{t\leq t_{\rm{follow}}}\left\lvert\frac{X_N(t)}{N}-x(t)\right\rvert>C_2N^{-\frac{1-h}{2}}\right)\leq4e^{-C_1N^h},
\end{equation}
where 
\begin{equation}\label{tfollow}
    t_{\rm{follow}}\equiv t_{\rm{follow}}(N, h):=\mbox{$\frac{1}{J}$}\lceil e^{C_1N^{h}}\rceil.
\end{equation}
\end{lemma}
As will be seen in the proof, it is enough to consider $h$ fixed (say equal to $1/2$), 
so the dependence of~$N$ and $t_{\rm{follow}}$ on $h$ is ultimately not important.

Our other result in this section controls the mean of $X_N(t)/N$, and is used in the proof of the lower bound on the mixing time in Section \ref{lowerSIS}.

\begin{lemma}\label{mean}
Let $\mathbb{E}_{x_0}$ be the expectation given the starting state $X_N(0)=x_0$. The following two statements hold:
\begin{enumerate}
    \item[(1)] There exists a positive constant $K_1$, depending only on $\lambda,$ $\mu$ and $\varepsilon$, such that, for any $c\in\mathbb{R}$ and $N$ large enough (depending on $\lambda,\mu,\varepsilon$, and $c$),  
    \begin{equation*}
        \left\lvert \mathbb{E}_{x_0}\left(X_N\left(\mbox{$\frac{1}{2J}$}\log(N)+c\right)\right)-x^{\star}N \right\rvert\leq K_1 e^{-Jc}\sqrt{N}.
    \end{equation*}
    \item[(2)]  Suppose that that $x_0/N-x^{\star} > \alpha$, for some $\alpha > 0$. Then there exists a positive constant $K_2$, depending on $\lambda, \mu, \varepsilon$ and $\alpha$, such that, for any $c\in\mathbb{R}$ and $N$ large enough (depending on $\lambda, \mu, \varepsilon, \alpha$, and $c$), 
    \begin{equation*}
        K_2 e^{-Jc}\sqrt{N}\leq \mathbb{E}_{x_0}\left(X_N\left(\mbox{$\frac{1}{2J}$}\log(N)+c\right)\right)-x^{\star}N.
    \end{equation*}
\end{enumerate}
\end{lemma}

\subsection{Proof of Lemma~\ref{following}}
We study the centred processes $Y_N(t):=X_N(t)/N - x^{\star}$ and $y(t):=x(t)-x^{\star}$,  with $x(0)=X_N(0)/N$. 
The overall strategy of the proof is to give an integral representation of the deterministic process $y(t)$ (Lemma~\ref{lem:altrepxt}), which can be compared to a  similar representation of $Y_N(t)$, but with an addition
of a martingale term
 (Lemma~\ref{yrep}). From here, the result essentially follows by controlling the martingale term (Lemma~\ref{T_1bound}), using a simplified version of \cite[Lemma~7]{lopes2020extinction} (stated here in Lemma~\ref{expmartingalebound})
 and then applying Gr\"onwall’s lemma.

We first give a representation of $y(t)$.
\begin{lemma}\label{lem:altrepxt}
    In the notation just defined, we have 
    \begin{equation}\label{yint}
        y(t)=y(0)-\int_0^t \lambda y^2(s)+Jy(s)\,ds,
    \end{equation}
    and 
    \begin{equation}\label{expint}
        y(t)=e^{-Jt}y(0)-\int_0^t \lambda e^{-J(t-s)} y^2(s)\,ds.
    \end{equation}
\end{lemma}
\begin{proof}
    Using the definition of $x^{\star}$, the governing equation \eqref{epssis} can be rewritten as
    \begin{equation*}
        \frac{dx}{dt}=-\lambda(x-x^{\star})\left(x-x^{\star}+\frac{J}{\lambda}\right),
    \end{equation*}
    which is the same as
    \begin{equation}\label{ydiff}
        \frac{dy}{dt}=-\lambda y^2-Jy.
    \end{equation}
From here, \eqref{yint} follows by integration. Moreover, a simple calculation using \eqref{ydiff} shows that
    \begin{equation*}
        \frac{d}{ds}e^{-J(t-s)}y(s)=-\lambda e^{-J(t-s)} y^2(s),
    \end{equation*}
    which is the same as~\eqref{expint} after integrating.
\end{proof}
We compare $Y_N(t)=X_N(t)/N-x^{\star}$ to $y(t)$ via a representation that is similar to \eqref{expint}.
\begin{lemma}\label{yrep}
In the notation just defined, we have
    \begin{equation}\label{expsto}
    Y_N(t)=e^{-Jt}Y_N(0)-\int_0^t \lambda e^{-J(t-s)}Y^2_N(s)\,ds+\int_0^t e^{-J(t-s)}\,dM_N(s),
\end{equation}
where 
\begin{equation}\label{M_N}
    M_N(t):=Y_N(t)-Y_N(0)+\int_0^t \lambda Y^2_N(s)-JY_N(s)\,ds,
\end{equation}
is a zero-mean martingale.
\end{lemma}
\begin{proof}
To show that the process $M_N(t)$ defined in~\eqref{M_N} is a martingale, note that the family of Markov chains $(X_N(t))_{N\geq1}$ is a \textit{density dependent family} (see Section 4 in \cite{kurtz1971limit}) with transitions of the form
    \begin{align*}
    &x\mapsto x+1\quad \text{at rate}\quad Nf\left(\frac{x}{N}, +1\right):= N\cdot\left(\lambda \frac{x}{N}\left(1-\frac{x}{N}\right)+\varepsilon\left(1-\frac{x}{N}\right)\right),\\
    &x\mapsto x-1\quad \text{at rate}\quad Nf\left(\frac{x}{N}, -1\right):=N\cdot\mu\frac{x}{N},
    \end{align*}
    and the state space $\{0, 1,\ldots, N\}$ is compact. It is easy to check that the scaled process $(X_N(t)/N)_{t\geq0}$ satisfies the conditions of \cite[Proposition 2.1]{kurtz1971limit} and thus
    \begin{equation*}
        M_N(t):=\frac{X_N(t)}{N}-\frac{X_N(0)}{N}-\int_0^tF\left(\frac{X_N(s)}{N}\right)\,ds,
    \end{equation*}
    is a zero-mean martingale, where $F(x)=f\left(x, +1\right)-f\left(x, -1\right)$. 
    Then formula~\eqref{expsto} holds, by \cite[Lemma~4.1]{barbour2012law}.   
\end{proof}
Comparing expressions \eqref{expint} and \eqref{expsto}, it is clear that to be able to approximate $Y_N(t)$ by $y(t)$ we must control the size of $\big\lvert\int_0^t e^{-J(t-s)}\,dM_N(s)\big\rvert$. To do this,  we use the following simplified version of~\cite[Lemma~7]{lopes2020extinction} for multi-dimensional Markov jump chains.
\begin{lemma}\cite[Lemma 7]{lopes2020extinction}\label{expmartingalebound}
Let $(X(t))_{t\geq0}$ be a pure jump Markov chain with state space $E\subseteq\mathbb{R}$. For $x,l\in\mathbb{R}$ such that $x, x+l \in E$, let $q(x, x+l)$ denote the rate the chain jumps from $x$ to $x+l$, and assume that there is a deterministic $B>0$ such that $q(x, x+l)=0$ for $\lvert l \rvert >B$. Suppose further that, for each $x\in E$, the drift $F(x):=\sum_{l}lq(x, x+l)$ at $x$ can be written in the form
\begin{equation*}
    F(x)=\widetilde{A}x+\widetilde{F}(x), \quad \text{for some } \widetilde{A}\leq 0.
\end{equation*}
Let $(M(t))_{t\geq0}$ be the corresponding Dynkin martingale given by
\begin{equation*}
    M(t):=X(t)-X(0)-\int_0^t F(X(s))\,ds,
\end{equation*}
and let $\widetilde{M}(t)=\int_0^t e^{\widetilde{A}(t-s)}\,dM(s)$. For $\delta>0$, define $\widetilde{T}(\delta)=\inf\{t\geq0:\lvert \widetilde{M}(t)\rvert>\delta\}$ to be  the infimum of time $t$ such that $\widetilde{M}(t)$ exceeds $\delta$ in absolute value. Further, given $u\in\mathbb{R}_+$, let $\nu(z, u):=\sum_w q(z, z+w)(e^{\widetilde{A}u}w)^2$, and suppose that for some $K>0$, $\int_0^t \nu(X(s), t-s)\,ds\leq K$. 
Then for any $\sigma>0$ and $0<\omega\leq 4(\log2)^2K/B^2$, we have
\begin{equation*}
    \mathbb{P}(\widetilde{T}(e^{-\widetilde{A}\sigma}\sqrt{\omega K})\leq \sigma \lceil e^{\omega/8} \rceil)\leq 4e^{-\omega/8}.
\end{equation*}
\end{lemma}

Specialising the lemma to our setting, we obtain the following result, which  controls  the size of the term $\big\lvert\int_0^t e^{-J(t-s)}\,dM_N(s)\big\rvert$.
\begin{lemma}\label{T_1bound}
    Fix $h\in(0,1)$, set $\omega:=4(\log2)^2kN^h$, $k:=(\lambda+\mu+\varepsilon)/2J$, and 
    \begin{equation*}
        T_1:=\inf\left\{t\geq0: \left\lvert\int_0^t e^{-J(t-s)}\,dM_N(s)\right\rvert>e\sqrt{\frac{\omega k}{N}}\right\}.
    \end{equation*}
Then
    \begin{equation*}
        \mathbb{P}\left(T_1\leq \frac{1}{J}\left\lceil e^{\omega/8}\right\rceil\right)\leq 4e^{-\omega/8}.
    \end{equation*}
\end{lemma}
\begin{proof}
    We apply Lemma \ref{expmartingalebound} to our process $Y_N(t)$. Clearly, the only possible jumps for $Y_N(t)$ are $\pm 1/N$ so we can take $B:=1/N$.  The drift $F_X$ of $X_N$ is given by 
    \begin{equation*}
        F_X(x)=1\cdot\lambda x\left(1-\frac{x}{N}\right)+\varepsilon(N-x)+(-1)\cdot \mu x=-\lambda N\left(\frac{x}{N}-x^{\star}\right)\left(\frac{x}{N}-x^{\star}+\frac{J}{\lambda}\right).
    \end{equation*}
    Hence it is easily seen that
the drift of $Y_N$ is given by $F_Y(y)=\widetilde{A}y+\widetilde{F}(y)$, where $\widetilde{A}=-J$ and $\widetilde{F}(y)=-\lambda y^2$. 

Moreover, there is a uniform bound on $\int_0^t\nu(Y_N(s), t-s)ds$, where $\nu(z, u):=\sum_w q(z, z+w)(e^{-Ju}w)^2$ and $q(y, y+w)$ is the rate of the jump from $y$ to $y+w$ for $y\in\{0, 1, \ldots, N\}/N$, $w = \pm 1/N$:
    \allowdisplaybreaks
    \begin{align*}
        &\int_0^t\nu(Y_N(s), t-s)\,ds\nonumber\\
        &\quad=\int_0^t\left(\lambda X_N(s)\left(1-\frac{X_N(s)}{N}\right)+\varepsilon\left(N-X_N(s)\right)+\mu X_N(s)\right)e^{-2J(t-s)}N^{-2}\,ds\nonumber\\
        &\quad=\frac{1}{N}\int_0^t\left(\lambda \frac{X_N(s)}{N}\left(1-\frac{X_N(s)}{N}\right)+\varepsilon\left(1-\frac{X_N(s)}{N}\right)+\mu\frac{X_N(s)}{N}\right)e^{-2J(t-s)}\,ds\nonumber\\
        &\quad\leq \frac{1}{N}\int_0^t\left((\lambda+\mu)\frac{X_N(s)}{N}+\varepsilon\right)e^{-2J(t-s)}\,ds\nonumber\\
        &\quad\leq \frac{\lambda+\mu+\varepsilon}{N}\int_0^te^{-2J(t-s)}\,ds\leq\frac{k}{N}.
    \end{align*}
    The result then follows directly from the  Lemma \ref{expmartingalebound} by setting $\sigma=\frac{1}{J}$ and $K=k/N$.
\end{proof}
The last missing ingredient for the proof of Lemma \ref{following} is a quantitative upper bound on the speed of convergence for the solution to the governing equation~\eqref{epssis} with initial condition $x(0)=X_N(0)/N$ to the fixed point $x^{\star}$. This will then be used to bound 
$\lvert Y_N(t)-y(t) \rvert= \lvert X_N(t)-x(t) \rvert$. The following lemma with an elementary calculus proof gives such a bound.
\begin{lemma}\label{detb}
Using the notation defined above, the following inequality holds, 
\begin{equation}\label{yboundpart3}
    \lvert y(t) \rvert\leq\frac{2J}{J-(\lambda-\mu-\varepsilon)}\lvert y(0)\rvert e^{-tJ}, \quad t\geq0.
\end{equation}
\end{lemma}
\begin{proof}
Recall that $x^{\star}=\frac{\lambda-\mu-\varepsilon+J}{2\lambda}\geq 0, x_1^{\star}=\frac{\lambda-\mu-\varepsilon-J}{2\lambda}\leq 0$, where $J=\sqrt{(\lambda-\mu-\varepsilon)^2+4\lambda\varepsilon}$. We have $-x^{\star} \leq y(0) \leq 1-x^{\star}$. Also, from \eqref{parsolepssis} and the fact $J/\lambda\geq x^{\star}$, we obtain
\begin{align*}
    \lvert y(t) \rvert
    &=\left\lvert\frac{J}{2\lambda}\frac{2 y(0)}{(x^{\star}-x_1^{\star}+y(0))e^{tJ}-y(0)}\right\rvert
    =\frac{\frac{J}{\lambda}\lvert y(0) \rvert}{\left(\frac{J}{\lambda}+y(0)\right)e^{tJ}-y(0)} \\
    &=\frac{\frac{J}{\lambda}\lvert y(0)\rvert}{\frac{J}{\lambda}+(1-e^{-tJ}) y(0)}e^{-tJ}.
\end{align*}
It then follows by easy manipulation that
\begin{equation*}
    \lvert y(t) \rvert\leq \frac{\frac{J}{\lambda}\lvert y(0)\rvert }{\frac{J}{\lambda}-(1-e^{-tJ})x^{\star}}e^{-tJ}
    \leq\frac{\frac{J}{\lambda}\lvert y(0)\rvert }{\frac{J}{\lambda}-\frac{\lambda-\mu-\varepsilon+J}{2\lambda}}e^{-tJ}
    =\frac{2J}{J-(\lambda-\mu-\varepsilon)}\lvert y(0)\rvert e^{-tJ}.
\qedhere\end{equation*}
\end{proof}
We can now prove Lemma \ref{following}. As previously discussed, we use representations \eqref{expint} and \eqref{expsto} to bound the difference between $X_N(t)/N$ and $x(t)$, and use Lemma \ref{T_1bound} to control the martingale term in \eqref{expsto}.
\begin{proof}[Proof of Lemma \ref{following}]
Using \eqref{expint} and \eqref{expsto}, we have
\begin{align*}
    \widetilde{e}_N(t)&:=\lvert Y_N(t)-y(t)\rvert \\
    &\leq\lambda\int_0^t e^{-J(t-s)}\left\lvert Y_N(s)^2-y(s)^2\right\rvert\,ds+\left\lvert\int_0^t e^{-J(t-s)}\,dM_N(s)\right\rvert\\
    &\leq\lambda\int_0^t e^{-J(t-s)} \widetilde{e}_N(s)(\widetilde{e}_N(s)+2\lvert y(s)\rvert)\,ds+\left\lvert\int_0^t e^{-J(t-s)}\,dM_N(s)\right\rvert.
\end{align*}
Using \eqref{yboundpart3} to bound $\lvert y(s) \rvert$, we obtain
\begin{align}\label{lowrangeerror}
    \begin{split}
        \widetilde{e}_N(t)\leq\frac{4\lambda J}{J-(\lambda-\mu-\varepsilon)}\lvert y(0) \rvert e^{-Jt}\int_0^t\widetilde{e}_N(s)\,ds +\lambda\int_0^t & e^{-J(t-s)} \widetilde{e}_N(s)^2\,ds\\
        &+\left\lvert\int_0^t e^{-J(t-s)}\,dM_N(s)\right\rvert.
    \end{split}
\end{align}
To bound the first two terms on the right-hand side of \eqref{lowrangeerror}, we define 
\begin{equation*}
    T_2:=\inf\left\{t:\widetilde{e}_N(t)>3\widetilde{C}\sqrt{\frac{\omega k}{N}}\right\},
\end{equation*}
where $\widetilde{C}:=4e^{\frac{4\lambda}{J-(\lambda-\mu-\varepsilon)}(x^{\star}\vee (1-x^{\star}))}$ (the exact form is not so important). Recall the definition of $T_1$, $\omega$ and $k$ from Lemma \ref{T_1bound}. We will show that on the event $t\leq T_1 \wedge T_2$, one has
\begin{equation}\label{tildeebound}
    \widetilde{e}_N(t)\leq\widetilde{C}\sqrt{\frac{\omega k}{N}}.
\end{equation}
For now, suppose this is true. Define
\begin{equation*}
    T_3:=\inf\left\{t:\widetilde{e}_N(t)>2\widetilde{C}\sqrt{\frac{\omega k}{N}}\right\}.
\end{equation*}
By definition, $\mathbb{P}(T_3 \leq T_2)=1$. Moreover, since we have assumed that \eqref{tildeebound} holds, we also have $\mathbb{P}(T_3\leq T_1 \wedge T_2)=0$, which implies that $0=\mathbb{P}(T_3\leq T_1\wedge T_2)=\mathbb{P}(T_3\leq T_1)$.  Applying Lemma \ref{T_1bound}, we obtain 
\begin{equation*}
    \mathbb{P}\left(T_3\leq\mbox{$\frac{1}{J}$}\lceil e^{\omega/8}\rceil\right)\leq\mathbb{P}\left(T_1\leq\mbox{$\frac{1}{J}$}\lceil e^{\omega/8}\rceil\right)\leq4e^{-\omega/8},
\end{equation*}
which completes the proof. The only thing left to check is \eqref{tildeebound}. On the event $t\leq T_1$, it follows from \eqref{lowrangeerror} that
\begin{equation*}
    \widetilde{e}_N(t)\leq\frac{4\lambda J}{J-(\lambda-\mu-\varepsilon)}\lvert y(0) \rvert e^{-Jt}\int_0^t \widetilde{e}_N(s)\,ds+\lambda\int_0^t e^{-J(t-s)} \widetilde{e}_N(s)^2\,ds+e\sqrt{\frac{\omega k}{N}}.
\end{equation*}
To further bound $\widetilde{e}_N(t)$, a standard method is to use the Gr\"onwall’s Lemma (see for example Theorem 1.3.1 in 
\cite{ames1997inequalities}). However, to apply the inequality we first need to control the $\widetilde{e}_N^2$ term. On the event $t\leq T_1 \wedge T_2$, one has
\begin{align*}
    \widetilde{e}_N(t)&\leq \frac{4\lambda J}{J-(\lambda-\mu-\varepsilon)}\lvert y(0) \rvert e^{-tJ}\int_0^t \widetilde{e}_N(s)\,ds+9\lambda\widetilde{C}^2\frac{\omega k}{N}\int_0^t e^{-J(t-s)}\,ds+e\sqrt{\frac{\omega k}{N}}\\
    &\leq\frac{4\lambda J}{J-(\lambda-\mu-\varepsilon)}\lvert y(0) \rvert e^{-tJ}\int_0^t \widetilde{e}_N(s)\,ds+\frac{9\lambda\widetilde{C}^2}{J}\frac{\omega k}{N}+e\sqrt{\frac{\omega k}{N}}.
\end{align*}
Choose $N>(18k\lambda\widetilde{C}^2J^{-1}\log{2} )^{\frac{2}{1-h}}$, so that 
\begin{equation}\label{eq:11}
    \frac{9\lambda\widetilde{C}^2}{J}\frac{\omega k}{N}=\frac{18k\lambda\widetilde{C}^2\log{2} }{J}N^{-\frac{1-h}{2}}\sqrt{\frac{\omega k}{N}}<\sqrt{\frac{\omega k}{N}},
\end{equation}
and hence
\begin{equation*}
    \widetilde{e}_N(t)\leq\frac{4\lambda J}{J-(\lambda-\mu-\varepsilon)}\lvert y(0) \rvert e^{-Jt}\int_0^t \widetilde{e}_N(s)\,ds+4\sqrt{\frac{\omega k}{N}}.
\end{equation*}
Let $\widehat{e}_N(t)=e^{tJ}\widetilde{e}_N(t)$, then on the event $t\leq T_1 \wedge T_2$, we have
\begin{equation*}
    \widehat{e}_N(t)\leq\frac{4\lambda J}{J-(\lambda-\mu-\varepsilon)}\lvert y(0) \rvert\int_0^t e^{-sJ}\widehat{e}_N(s)\,ds+4e^{tJ}\sqrt{\frac{\omega k}{N}},
\end{equation*}
for $N$ large enough. Applying Gr\"onwall’s lemma to $\widehat{e}_N(t)$, we obtain for $t\leq T_1 \wedge T_2$ that
\begin{align*}
    \widehat{e}_N(t)\leq4e^{\frac{4\lambda J}{J-(\lambda-\mu-\varepsilon)}\lvert y(0) \rvert\frac{1-e^{-tJ}}{J}}\sqrt{\frac{\omega k}{N}}e^{tJ}&\leq4e^{\frac{4\lambda}{J-(\lambda-\mu-\varepsilon)}\lvert y(0) \rvert}\sqrt{\frac{\omega k}{N}}e^{tJ}
\end{align*}
Dividing $e^{tJ}$ from both side the inequality above then establishes \eqref{tildeebound} as $\lvert y(0) \rvert \leq x^{\star}\vee (1-x^{\star})$. 
\end{proof}

\subsection{Proof of Lemma~\ref{mean}}
To explain the proof of the other main result of this section, write $y_0=x_0/N-x^{\star}$, which is the corresponding starting state of the chain $Y_N(t)$. Recall the definition of $M_N(t)$ in \eqref{M_N}, Since $\mathbb{E}(M_N(t))=0$ for all $t\geq0$, we have
    \begin{align}\label{yb}
        \mathbb{E}_{x_0}(Y_N(t))&=y_0-\mathbb{E}_{x_0}\left(\int_0^t JY_N(s)\,ds\right)-\mathbb{E}_{x_0}\left(\int_0^t \lambda Y^2_N(s)\,ds\right)\nonumber\\
        &= y_0-\int_0^t J\mathbb{E}_{x_0}\left(Y_N(s)\right)\,ds-\int_0^t \lambda\mathbb{E}_{x_0}\left( Y^2_N(s)\right)\,ds,
    \end{align}
    where we have used Fubini in the last equality to change the order of expectation and the integral. If we could replace $\mathbb{E}_{x_0}\left( Y^2_N(s)\right)$ by $\left(\mathbb{E}_{x_0}Y_N(s)\right)^2$, then $\mathbb{E}_{x_0}\left(Y_N(s)\right)$ would exactly solve the  integral equation~\eqref{expint}, and the claim of the lemma would be trivial. Thus we prove the result by showing the difference between $\mathbb{E}_{x_0}\left( Y^2_N(t)\right)$ and $\left(\mathbb{E}_{x_0}Y_N(t)\right)^2$ is small for all $t\leq t_{\rm{follow}}$, and so the difference between $\int_0^t \lambda\mathbb{E}_{x_0}\left( Y^2_N(s)\right)ds$ and $\int_0^t \lambda\left(\mathbb{E}_{x_0}Y_N(s)\right)^2ds$ is also small for not too large~$t$.

\begin{proof}[Proof of Lemma \ref{mean}]    
    Let $y(t)$ solve~\eqref{ydiff} and $A_t:=\{z:\lvert z-y(t)\rvert \leq C_2 N^{-\frac{1-h}{2}}\}$ for fixed $h< 1/2$ and $C_2$ from Lemma~\ref{following}. Applying that lemma with this $h$, we conclude that $\mathbb{P}_{x_0}(Y_N(t)\in A_t^c)$ is exponentially small in $N$ for  $t\leq t_{\rm{follow}}$. Together with the fact that $Y_N(t)$ is bounded for all $t\geq0$, we obtain
    \begin{equation*}
        \mathbb{E}_{x_0}(Y_N(t))=\mathbb{E}_{x_0}(Y_N(t)\mathbbm{1}_{Y_N(t)\in A_t})+\mathbb{E}_{x_0}(Y_N(t)\mathbbm{1}_{Y_N(t)\in A_t^c})=y(t)+O(N^{-\frac{1-h}{2}}),
    \end{equation*}
    and using this and a similar argument implies
    \begin{align}\label{varbound}   \mathrm{Var}_{x_0}\left(Y_N(t)\right)&=\mathbb{E}_{x_0}\left(\left(Y_N(t)-\mathbb{E}_{x_0}(Y_N(t))\right)^2\mathbbm{1}_{Y_N(t)\in A_t}\right)\nonumber\\
    &\qquad\qquad+\mathbb{E}_{x_0}\left(\left(Y_N(t)-\mathbb{E}_{x_0}(Y_N(t))\right)^2\mathbbm{1}_{Y_N(t)\in A_t^c}\right)=O(N^{-(1-h)}),
    \end{align}
    for  $t\leq t_{\rm{follow}}$.
    Rewriting \eqref{yb} to take advantage of these bounds, we have
    \begin{equation}\label{eq:inteq22}
        \mathbb{E}_{x_0}(Y_N(t))
        	=y_0-\int_0^t J\mathbb{E}_{x_0}\left(Y_N(s)\right)\,ds
		-\int_0^t \lambda\left( \mathbb{E}_{x_0}Y_N(s)\right)^2\, ds-\int_0^t \mathrm{Var}_{x_0}\left(Y_N(s)\right)\,ds.
    \end{equation}
    With this equation and the bound~\eqref{varbound}, we can use standard differential equation comparison theorems to bound $\mathbb{E}_{x_0}(Y_N(t))$. Let $u(t):=\mathbb{E}_{x_0}(Y_N(t))$ and differentiate the integral equation~\eqref{eq:inteq22} to get the initial value problem 
    \begin{equation}\label{eq:u}
        \frac{d}{dt}u(t)= -\lambda u^2(t)-Ju(t)-\mathrm{Var}_{x_0}\left(Y_N(t)\right), \quad u(0)=y_0, \quad t\leq t_{\rm{follow}}.
    \end{equation}
    It follows from \eqref{varbound} that there exists a suitable constant $C^*>0$ such that $\mathrm{Var}_{x_0}\left(Y_N(s)\right)\leq C^*N^{-(1-h)}$ and so we obtain 
    \begin{equation*}
        \frac{d}{dt}u(t)\geq -\lambda u^2(t)-Ju(t)-C^*N^{-(1-h)}, \quad u(0)=y_0, \quad t\leq t_{\rm{follow}}.
    \end{equation*}
    From  standard comparison theory of differential equations (see for example Theorem 2.2.2 in \cite{ames1997inequalities}), we have $u(t)\geq z(t)$ for all $t\geq0$, where $z(t)$ is the minimal solution to the initial value problem    
    \begin{equation*}
        \frac{\,d}{dt}z(t)= -\lambda z^2(t)-Jz(t)-C^*N^{-(1-h)}, \quad z(0)=y_0, \quad t\leq t_{\rm{follow}}.
    \end{equation*}
    
    This differential equation is of the same form as the governing equation \eqref{epssis}, so it is not hard to see that for sufficiently large $N$ such that $J^2-4\lambda C^*N^{-(1-h)}>0$, the solution is uniquely defined for all $t\geq0$ and thus
    \begin{equation*}
        u(t)\geq z(t)=c_2 + \frac{\frac{c_1}{\lambda}(y_0-c_2)}{(y_0-c_3)-(y_0-c_2)e^{-c_1t}}e^{-c_1t},
    \end{equation*}
    where
    \begin{equation*}
        c_1:=\sqrt{J^2-4\lambda C^*N^{-(1-h)}}\in(0,J), \quad c_2:=\frac{-J+c_1}{2\lambda}<0, \quad  c_3:=\frac{-J-c_1}{2\lambda}<0.
    \end{equation*}
    For any $h < 1/2$, we can write $c_1=J-\delta$, $c_2=-\frac{\delta}{2\lambda}$, and $c_3=-\frac{J}{\lambda}+\frac{\delta}{2\lambda}$, where $\delta=o(N^{-\frac{1}{2}})>0$. Using this notation, the bound above then becomes
    \begin{equation}\label{eq: mean_lower_bound}
        u(t)\geq z(t)= -\frac{\delta}{2\lambda} + \frac{\frac{J-\delta}{\lambda}(y_0+\frac{\delta}{2\lambda})}{(y_0+\frac{J}{\lambda}-\frac{\delta}{2\lambda})e^{-\delta t}-(y_0+\frac{\delta}{2\lambda})e^{-Jt}}e^{-Jt}.
    \end{equation}
    Note that if $z(0)=y_0>-\frac{\delta}{2\lambda}$ then, as $t$ increases, $z(t)$ monotonically approaches $-\frac{\delta}{2\lambda}$ from above, whereas $z(t)$ monotonically approaches $-\frac{\delta}{2\lambda}$ from below if $z(0)=y_0\in(-\frac{2J-\delta}{2\lambda}, -\frac{\delta}{2\lambda})$. (Note that $y_0 \le -\frac{2J-\delta}{2\lambda}$ is not biologically viable.) Moreover, for any fixed $t$, $z(t)$ is monotonically increasing as $y_0$ increases, which can be verified by differentiating \eqref{eq: mean_lower_bound} with respect to $y_0$. 
    
    On the other hand, since $\mathrm{Var}_{x_0}\left(Y_N(t)\right)>0$, it follows from \eqref{eq:u} that 
    \begin{equation*}
        \frac{\,d}{dt}u(t)\leq -\lambda u^2(t)-Ju(t), \quad u(0)=y_0, \quad t\geq 0.
    \end{equation*}
    Again, by the comparison theorem and using~\eqref{ydiff}, we have
    \begin{equation}\label{eq: mean_upper_bound}
        u(t)\leq y(t)=\frac{\frac{J}{\lambda}y_0}{(y_0+\frac{J}{\lambda})-y_0e^{-Jt}}e^{-Jt}.
    \end{equation}
    Similar to the behaviour of $z(t)$, if $y(0)=y_0>0$ then $y(t)$ monotonically approaches $0$ from above, whereas $y(t)$ monotonically approaches $0$ from below if $y(0) \in ( - \frac{J}{\lambda}, 0)$. (Note that $y_0 \le -\frac{J}{\lambda}$ is not biologically viable.) Also, by differentiating \eqref{eq: mean_upper_bound} with respect to $y_0$, one can see that for any fixed $t$, $y(t)$ is monotonically increasing as $y_0$ increases. 
    
    For any $c\in\mathbb{R}$, we take $t=\frac{1}{2J}\log(N)+c$, which is then non-negative for $N$ sufficiently large. If in the right-hand side of the relation~\eqref{eq: mean_lower_bound}, one has $y_0=x_0/N-x^{\star}>\alpha>0$, then, since $e^{-\delta t} \to1$ and $e^{-Jt} \to0$ as $N\to\infty$, we obtain
    \begin{align*}
        u(t)\geq z(t) \geq \left(o(1)+\left( \frac{\alpha\frac{J}{\lambda}}{\alpha+\frac{J}{\lambda}} + o(1)\right)\right) \frac{e^{-Jc}}{\sqrt{N}}.
    \end{align*}
    Therefore, there exists a constant $K_2$ depending on $\lambda, \mu, \varepsilon$ and $\alpha$ such that for $N$ large enough,
    \begin{equation}\label{eq:mean_part_2}
        u(t)\geq z(t) \geq \frac{e^{-Jc}K_2}{\sqrt{N}},
    \end{equation}
    which finishes the proof of the second statement in Lemma~\ref{mean}. For the first statement, in view of~\eqref{eq: mean_lower_bound} and \eqref{eq: mean_upper_bound}, and the monotonicity of $y(t), z(t)$ with respect to the initial value $y_0$, one has, for any $0<y_0 \leq 1-x^{\star}$, 
    \begin{equation}\label{eq: utupper}
        -\frac{\delta}{2\lambda} < z(t) ,\qquad 0<y(t) \leq \bar{y}(t),
    \end{equation}
    where
    \begin{equation*}
        \bar{y}(t):=\frac{\frac{J}{\lambda}(1-x^{\star})}{(1-x^{\star}+\frac{J}{\lambda})-(1-x^{\star})e^{-Jt}}e^{-Jt}.
    \end{equation*}
    Similarly, for any $-x^{\star} \leq y_0<-\frac{\delta}{2\lambda}$, 
    \begin{equation}\label{eq: utlower}
        \underline{z}(t) \leq z(t) \leq  y(t) < 0,
    \end{equation}
    where
    \begin{equation*}
        \underline{z}(t):=-\frac{\delta}{2\lambda} + \frac{\frac{J-\delta}{\lambda}(-x^{\star}+\frac{\delta}{2\lambda})}{(-x^{\star}+\frac{J}{\lambda}-\frac{\delta}{2\lambda})e^{-\delta t}-(-x^{\star}+\frac{\delta}{2\lambda})e^{-Jt}}e^{-Jt}.
    \end{equation*}
    Also, observe that when $y_0\in[-\frac{\delta}{2\lambda},0]$, we have $z(t), y(t) \in [-\frac{\delta}{2\lambda},0]$ for all $t\geq0$. Since $\lvert u(t) \rvert \leq \lvert z(t) \rvert \vee \lvert y(t) \rvert$, it follows from \eqref{eq: utupper} and \eqref{eq: utlower} and the previous observation that
    \begin{equation*}
    \lvert u(t) \rvert\leq
        \begin{cases}
            -\underline{z}(t), & y_0\in [-x^{\star},-\frac{\delta}{2\lambda}),\\
            \frac{\delta}{2\lambda}, & y_0\in[-\frac{\delta}{2\lambda},0],\\
            \frac{\delta}{2\lambda}\vee\bar{y}(t), & y_0\in (0, 1-x^{\star}].
        \end{cases}
    \end{equation*}
    
    Now, fix any $h < 1/2$, so that $\delta=o(N^{-\frac{1}{2}})>0$. The proof is completed by noticing that for some positive constants $K_1^*, K_1'$ depending only on $\lambda, \mu$ and $\varepsilon,$ one has
    \begin{equation*}
        \bar{y}(t)\leq \frac{e^{-Jc}K_1^*}{\sqrt{N}}, \quad -\underline{z}(t)\leq \frac{e^{-Jc}K_1'}{\sqrt{N}},
    \end{equation*}
    for $N$ large enough.
\end{proof}
\section{Proof of the upper bound for the mixing time}\label{coupling}
In this section, we prove the upper bound in Theorem \ref{cutoff}, which in particular shows 
that, for a population of size $N$, the \textit{mixing time} of the Markov chain $X_N$ is asymptotically at most $\frac{1}{2J}\log N$. 

Recall $\rho_N(t)=\max_{x\in\{0,1,\ldots,N\}}\lVert P^t_N(x,\cdot)-\pi_N \rVert_{TV}$ and $t_N=\frac{1}{2J}\log(N)$ from Definition~\ref{cutoffphenomenon} and Theorem~\ref{cutoff}.
\begin{lemma}[Cutoff upper bound, rapid mixing]\label{rapidlymixing}
For all $\xi\geq0$, we have 
\begin{equation*}
    \limsup_{N\to\infty}\rho_N\left(t_N+\xi\right)\leq \psi_1(\xi),
\end{equation*}
where $\psi_1(\cdot)$ is a non-negative real valued function depending only on $\lambda, \mu$ and $\varepsilon$ and $\psi_1(\xi)=O(1/\sqrt{\xi})$ as $\xi\to\infty$.
\end{lemma}
Let $Z_N(t)$ and $W_N(t)$ be two copies of the process $X_N(t)$, with initial states $Z_N(0)\geq W_N(0)$, moving independently until the \textit{coalescence time}, which is defined as
\begin{equation*}
    \tau_{\text{couple}}:=\inf\{t: W_N(t)=Z_N(t)\}.
\end{equation*}
After the coalescence time, the states in the two copies are the same and move together forever after. Note that this defines a \textit{Markovian coupling}.

By standard results (see Corollary 5.5 in \cite{levin2017markov}), to prove Lemma \ref{rapidlymixing}, it is sufficient to show the following. Write $\mathbb{P}_{w_0,z_0}$ for the probability measure of the chain given $W_N(0)=w_0$ and $Z_N(0)=z_0$.
\begin{lemma}\label{rapidmixing}
For all $w_0, z_0\in \{0,\ldots,N\}$ with $w_0\leq z_0$, and all $\xi\geq 0$, we have 
\begin{equation*}
\limsup_{N\to\infty}\mathbb{P}_{w_0,z_0}\left(\tau_{\rm{couple}}>\mbox{$\frac{1}{2J}$}\log(N)+\xi\right)\leq\psi_1(\xi),
\end{equation*}
where $\psi_1(\cdot)$ is a non-negative real valued function depending only on $\lambda, \mu$ and $\varepsilon$ and $\psi_1(\xi)=O(1/\sqrt{\xi})$ as $\xi\to\infty$.
\end{lemma}
Before we dive into the details, we give an outline of the proof. As mentioned in the introduction, we will break up the analysis into the following phases. 
\begin{itemize}
    \item[] \textbf{Initial phase}: We use the concentration result from the Section \ref{priorconcen} to show that with high probability, by the time $\frac{1-h}{2J}\log(N)$, both copies $W_N(t), Z_N(t)$ will be in the \emph{interior} of a \emph{good set} consisting of a ball of size $O(N^{\frac{1+h}{2}})$ around $x^{\star}N$. We also prove that if the copies start anywhere in the interior of the good set, then with high probability they remain in the good set for a time period that is exponential in $N$.
    \item[] \textbf{Intermediate phase}: After both copies reach the interior of the good set, the independent coupling of the chains is contractive while the copies stay in the good set. As a consequence, after another $\frac{h}{2J}\log(N)$ time (plus a constant time), the distance between them will drop to $\sqrt{N}$ with high probability. 
    \item[] \textbf{Final phase}: Once the copies are within $\sqrt{N}$ of each other and given that they are still in the good set, we show the coalescence time of the copies is no greater than the time it takes for an unbiased random walk (with transition rate of order $N$, so on average, there are $N$ events happening in a unit time period) to hit $0$, which is at most a time of constant order with high probability.
\end{itemize}

Throughout the rest of this section, we use $\mathbb{P}_{x_0}(\cdot)$ to denote the underlying probability measure of the process $X_N(t)$ with starting state $X_N(0)=x_0$. 
\subsection{Initial phase}
We first show that with high probability uniformly over all starting states, by the time $\frac{1-h}{2J}\log(N)$, the process $X_N(t)$ will be within $O(N^{\frac{1+h}{2}})$ of the fixed point $x^{\star}N$. This will be done by a simple application of Lemma \ref{following}, on account of Proposition \ref{detsol}. 

To state the result, we first need to introduce some notation, which will be used frequently later. Consider intervals of the form 
\begin{equation}\label{def:interval}
    I(r):=\left[x^{\star}-r, x^{\star}+r\right], \quad r>0.
\end{equation}
We denote the first time that the scaled chain leaves an interval $I$ by
\begin{equation}\label{xexit}
    \tau^{X}_{\rm{exit}}(I):=\inf\{t\geq 0: X_N(t)/N\notin I\}.
\end{equation}
To prove the upper bound for the mixing time, we fix a $h\in(0,1)$ and consider the following set to be the good set,
\begin{equation}\label{goodset}
    \widehat{S}\equiv\widehat{S}(N):=I\left(\eta\right), \quad \eta\equiv\eta(N):=2(C_2+C_3)N^{-\frac{1-h}{2}}, 
\end{equation}
where $C_2$ was defined in Lemma~\ref{following} and $C_3:=\left(\frac{J}{\lambda} \frac{x^{\star}}{\lvert x_1^{\star}\rvert}\right) \vee (1-x^{\star})$. We also define the interior of the good set by $\widetilde{S}\equiv\widetilde{S}(N):=I\left(\frac{\eta}{2}\right)$. Moreover, to simplify the notation, we write $\tau^{X}_{\rm{exit}}$ for $\tau^{X}_{\rm{exit}}(\widehat{S})$, i.e. the first exit time of $\widehat{S}$. 

Recall that $t_{\rm{follow}}=\mbox{$\frac{1}{J}$}\lceil e^{C_1N^{h}}\rceil$ from \eqref{tfollow}, where $C_1$ was defined in Lemma~\ref{following}. We prove the following lemma, where the first part of the lemma serves as the initial phase and the second part controls the exit time of the interval $I(r)$
\begin{lemma}\label{burn_x}
\hfill
\begin{enumerate}
    \item[(1)] For any $h\in(0,1)$ and $x_0\in\{0,\ldots, N\}$, one has  
    \begin{equation*}
        \mathbb{P}_{x_0}\left(N^{-1}X_N\left(\frac{1-h}{2J}\log(N)\right)\in\widetilde{S}\right)\geq 1-4e^{-C_1N^h},
    \end{equation*}
    for $N$ sufficiently large such that $\frac{1-h}{2J}\log(N)\leq t_{\rm{follow}}$ and as per Lemma~\ref{following}.
    \item[(2)] For any $h\in(0,1)$, suppose $r=r(N)$ satisfies the following condition 
    \begin{equation}\label{cond:eta}
        \frac{r}{2C_2}N^{\frac{1-h}{2}}\geq 1, \quad \text{as} \quad N\to\infty,
    \end{equation}
    and $x_0/N\in I(\frac{r}{2})$. Then we have 
    \begin{equation*}
        \mathbb{P}_{x_0}(\tau^{X}_{\rm{exit}}(I(r))> t_{\rm{follow}})\geq 1-4e^{-C_1N^h}.
    \end{equation*} 
    In particular, $r=\eta(N)$ satisfies \eqref{cond:eta} and so we have $\mathbb{P}_{x_0}(\tau^{X}_{\rm{exit}}> t_{\rm{follow}})\geq 1-4e^{-C_1N^h}$ for all $x_0/N\in \widetilde{S}$.
\end{enumerate}
\end{lemma}

\begin{proof}
Recall the solution to the governing equation of the deterministic $\varepsilon$-SIS model given in \eqref{parsolepssis}. A simple calculation using Lemma~\ref{detb} shows that 
\begin{equation}\label{x(t)bound}
    \left\lvert x\left( \mbox{$\frac{1-h}{2J}$}\log(N)\right) - x^{\star}\right\rvert\leq C_3N^{-\frac{1-h}{2}},
\end{equation}
uniformly for $x(0)\in \{0,\ldots, N\}$. It follows from a union bound and Lemma \ref{following} that, for large enough $N$ such that $\frac{1-h}{2J}\log(N)\leq t_{\rm{follow}}$,
\begin{align*}
    \mathbb{P}_{x_0}&\left(\left\lvert X_N\left(\mbox{$\frac{1-h}{2J}$}\log(N)\right)-x^{\star}N\right\rvert > (C_2+C_3)N^{\frac{1+h}{2}}\right)\\  
    &\qquad\leq \mathbb{P}_{x_0}\left(\left\lvert X_N\left(\mbox{$\frac{1-h}{2J}$}\log(N)\right)-x\left(\mbox{$\frac{1-h}{2J}$}\log(N)\right)N\right\rvert > C_2N^{\frac{1+h}{2}} \right)\\
    &\qquad\leq\mathbb{P}_{x_0}\left(\sup_{t\leq t_{\rm{follow}}}\left\lvert\frac{X_N(t)}{N}-x(t)\right\rvert>C_2N^{-\frac{1-h}{2}}\right)\leq4e^{-C_1N^h},
\end{align*}
which completes the proof of the first statement. For the second statement, note that, given $x(0) \in I(\frac{r}{2})$, by the monotonicity of $x(t)$ in $t$, we have $x(t)\in I(\frac{r}{2})$ for all $t\geq 0$. The condition \eqref{cond:eta}, implies that $I(\frac{r}{2}+C_2N^{-\frac{1-h}{2}}) \subset I(r)$ for sufficiently large $N$. Moreover, it follows from Lemma \ref{following} that the probability that $X_N(t)/N$ is within $C_2N^{-\frac{1-h}{2}}$ of $x(t)$ for all $t\leq t_{\rm{follow}}$ is at least $1-4e^{-C_1N^h}$. Together with the previous observation, this implies that the probability that $X_N(t)/N \in I(\frac{r}{2}+C_2N^{-\frac{1-h}{2}}) \subset I(r)$  for all $t\leq t_{\rm{follow}}$ is at least $1-4e^{-C_1N^h}$. 
\end{proof}
The next corollary states the result of Lemma \ref{burn_x} in terms of the  two coupled copies $W_N(t)$ and $Z_N(t)$. That is, by the time $\frac{1-h}{2J}\log(N)$, both $W_N(t), Z_N(t)$ have entered the interior of the good set with high probability. Moreover, once they have reached the interior, they will then remain in the good set for at least $t_{\rm{follow}}$ time which is exponential in $N$ with high probability. Let
\begin{align}\label{texit}
    \tau_{\rm{exit}}:=\inf\left\{t\geq 0: W_N(t)/N \notin \widehat{S} \text{\, or \,} Z_N(t)/N \notin \widehat{S} \right\},
\end{align}
be the first time either copy leaves the good set. The corollary follows directly from Lemma \ref{burn_x} by applying a union bound.
\begin{corollary}\label{burn}
For any $h\in(0,1)$, the following holds for $N$ sufficiently large such that $\frac{1-h}{2J}\log(N)\leq t_{\rm{follow}}$ and as per Lemma~\ref{following}:
\begin{enumerate}
    \item[(1)] For any $w_0, z_0\in\{0,\ldots, N\}$, 
    \begin{equation*}
        \mathbb{P}_{w_0,z_0}\left( N^{-1}W_N\left(\mbox{$\frac{1-h}{2J}$}\log(N)\right)\in\widetilde{S}, N^{-1}Z_N\left(\mbox{$\frac{1-h}{2J}$}\log(N)\right)\in\widetilde{S}\right)\geq 1-8e^{-C_1N^h}.
    \end{equation*}
    \item[(2)] Suppose $w_0/N\in\widetilde{S}, z_0/N \in\widetilde{S}$, then we have $\mathbb{P}_{w_0,z_0}(\tau_{\rm{exit}}> t_{\rm{follow}})\geq 1-8e^{-C_1N^h}$.
\end{enumerate}
\end{corollary}

\subsection{Intermediate phase}
For this phase, we prove that, after the burn-in period in the first phase, the distance between the two copies will drop to $\sqrt{N}$ with high probability after another $\frac{h}{2J}\log N$ time (plus a constant time) by showing that the coupling is contracting if both $W_N(t)/N, Z_N(t)/N$ are in the good set $\widehat{S}$, provided that $N$ is large enough so that $J/2\lambda>\eta(N)$, where $\eta$ is as defined in \eqref{goodset}. To be precise, recall from \eqref{texit} that $\tau_{\rm{exit}}$ is the first time either copies exits the good set $\widehat{S}$. Also, let $D_N(t):=Z_N(t)-W_N(t)$ be the difference between the two copies of the chain at time $t$. The lemma below states that, for any $\xi>0$, given that $D_N(0)=d_0$ is of order $N^{\frac{1+h}{2}}$, with high probability, on the event $\{\tau_{\rm{exit}}>\frac{h}{2J}\log(N)+\frac{\xi}{2}\}$, $D_N(t)$ drops below $\sqrt{N}$ by the time $\frac{h}{2J}\log(N)+\frac{\xi}{2}$.

For use in the proof, the jump rates of the  coupling $(W_N(t), Z_N(t))$ for $t\leq \tau_{\rm{couple}}$ are given by
\begin{equation}\label{couplingrule}
    \begin{aligned}
        &(w,z)\mapsto (w+1,z)&&\text{ at rate\quad} \mathcal{I}_w,\\
        &(w,z)\mapsto (w,z+1)&&\text{ at rate\quad} \mathcal{I}_z,\\
        &(w,z)\mapsto (w-1,z)&&\text{ at rate\quad} \mathcal{C}_w,\\
        &(w,z)\mapsto (w,z-1)&&\text{ at rate\quad} \mathcal{C}_z,
    \end{aligned}
\end{equation}
where 
\begin{equation*}
    \mathcal{I}_x\equiv \mathcal{I}(x):= \lambda x\left(1-\frac{x}{N}\right)+\varepsilon (N-x) \quad \text{and} \quad \mathcal{C}_x \equiv \mathcal{C}(x):= \mu x, \quad x\in\{0, \ldots, N\}.
\end{equation*}  
Note that before colliding the two copies almost surely do not jump simultaneously, and thus they do not cross without colliding. Since the chains move together after colliding and we assume (without loss of generality) that $W_N(0)\leq Z_N(0)$, the coupling is \textit{monotonic}, meaning $W_N(t)\leq Z_N(t)$ for all $t\geq 0$.

\begin{lemma}\label{med}
For $h\in(0,1)$, let $d_0:=z_0-w_0$. Suppose $w_0/N, z_0/N\in \widehat{S}$, so that $d_0\leq 2\eta N=4(C_2+C_3)N^{\frac{1+h}{2}}$. There exists a positive constant $C_4$ depending only on $\lambda, \mu$ and $\varepsilon$ such that the following relation holds for $N$ sufficiently large and any $\xi>0$,
\begin{equation}\label{hit}    
\mathbb{P}_{w_0,z_0}\left(D_{N}\left(\mbox{$\frac{h}{2J}$}\log(N)+\mbox{$\frac{\xi}{2}$}\right)\geq \sqrt{N}, \, \tau_{\rm{exit}}>\mbox{$\frac{h}{2J}$}\log(N)+\mbox{$\frac{\xi}{2}$}\right)\leq C_4e^{-\mbox{$\frac{J}{2}$}\xi}.
\end{equation}
\end{lemma}
\begin{proof}
Recall the possible transitions for this phase defined in \eqref{couplingrule}. Since the coupling is monotonic,
$D_N(t)$ is non-negative for all $t\geq0$. Also, $\tau_{\text{couple}}=\inf\{t\geq0:D_N(t)=0\}$. Note that the process $D_N(t)$ is not a Markov process by itself, but $D_N(t)$ is uniquely determined by the two-dimensional process $(W_N(t), Z_N(t))$, which is Markov with respect to its natural filtration. 

For each $N$, Proposition 2.1 in \cite{kurtz1971limit} implies the process
\begin{equation}\label{2dcoupling}
    \left(W_N(t), Z_N(t)\right)^{\top}-(W_N(0), Z_N(0))^{\top}-\int_0^t F\left((W_N(s), Z_N(s))^{\top}\right)\,ds, \quad t\geq 0,
\end{equation}
is a zero-mean martingale, where 
\begin{align*}
     F((w, z)^{\top})& := 
     \begin{cases}
     (0,1)^{\top} \mathcal{I}_z+ (1,0)^{\top}\mathcal{I}_w+ (0,-1)^{\top}\mathcal{C}_z + (-1,0)^{\top}\mathcal{C}_w, & w\not=z \\
     (1,1)^{\top} \mathcal{I}_w +(-1,-1)^{\top}\mathcal{C}_w, & w = z.
     \end{cases}
\end{align*}
Taking expectations in \eqref{2dcoupling} and subtracting the coordinates, we obtain 
\begin{align*}
    \mathbb{E}_{w_0, z_0}D_N(t)=(z_0-w_0) &+ \mathbb{E}_{w_0, z_0} \int_0^t \left(\mathcal{I}(Z_N(s))-\mathcal{I}(W_N(s))\right) \,ds\\ 
    &- \mathbb{E}_{w_0, z_0} \int_0^t (\mathcal{C}(Z_N(s))-\mathcal{C}(W_N(s))) \,ds.  
\end{align*}
Note that, for all $t\leq \tau_{\rm{exit}}$, one has $\left(x^{\star}-\eta\right)N\leq W_N(t)\leq Z_N(t)$. Moreover, for any states $w,z \in\widehat{S}$, we have
\begin{align*}
    \mathcal{I}_z-\mathcal{I}_w
    &=(\lambda-\varepsilon)(z-w)-\frac{\lambda}{N}(z-w)(z+w)\nonumber\\
    &\leq (\lambda-\varepsilon)(z-w)-\frac{2\lambda}{N}(z-w)\left(x^{\star}-\eta\right)N\nonumber\\
    &=\left(\mu-J+2\lambda\eta\right)(z-w).
\end{align*}
Choose $N$ large enough so that $J/2\lambda\geq \eta(N)$.
Applying the optional stopping theorem to the Dynkin martingale for $D_N(t)$ (which is a function of the two-dimensional Markov chain $(W_N(t),Z_N(t))$) and bounded stopping time $t \wedge \tau_{\rm{exit}}$, as well as Fubini's theorem, and the fact that $D_N(t)$ is non-negative, we obtain
\begin{align*}
    \mathbb{E}_{w_0, z_0}\left(D_N(t)\mathbbm{1}_{t< \tau_{\rm{exit}}}\right) &\leq \mathbb{E}_{w_0, z_0}D_N(t\wedge \tau_{\rm{exit}})\\
    &\leq(z_0-w_0) - (J-2\lambda\eta) \mathbb{E}_{w_0, z_0} \left(\int_0^{t \wedge \tau_{\rm{exit}}} D_N(s) \,ds\right)\\
    &=(z_0-w_0) - (J-2\lambda\eta) \mathbb{E}_{w_0, z_0} \left(\int_0^{t} D_N(s)\mathbbm{1}_{s<\tau_{\rm{exit}}} \,ds\right)\\
    &=(z_0-w_0) - (J-2\lambda\eta)\int_0^{t} \mathbb{E}_{w_0, z_0}\left(D_N(s)\mathbbm{1}_{s<\tau_{\rm{exit}}} \right)\,ds.
\end{align*}
Letting $d(t):=\mathbb{E}_{w_0, z_0}\left(D_N(t)\mathbbm{1}_{t< \tau_{\rm{exit}}}\right)$, we see that $d(t)$ satisfies the following integral inequality: 
\begin{equation*}
    d(t)\leq d_0 - (J-2\lambda\eta) \int_0^t d(s) \,ds. 
\end{equation*}
Then, for all $N$ large enough so that $J/2\lambda\geq \eta(N)$, by adapting Bellman's proof of the Gr\"{o}nwall's inequality (see Theorem 1.2.2 in \cite{ames1997inequalities}) to this setting, we obtain that for all $t\geq0$,
\begin{equation*}
    d(t)\leq d_0e^{-(J-2\lambda\eta)t}.
\end{equation*}
Since we have assumed that $d_0\leq 4(C_2+C_3)N^{\frac{1+h}{2}}$, we have
 \begin{align*}
    \mathbb{E}_{w_0, z_0}\left(D_N\left(\mbox{$\frac{h}{2J}$}\log(N)+\mbox{$\frac{\xi}{2}$}\right)
    \mathbbm{1}_{\tau_{\rm{exit}}>\mbox{$\frac{h}{2J}$}\log(N)+\mbox{$\frac{\xi}{2}$}}\right)&\leq d_0e^{-(J-2\lambda\eta) \left(\mbox{$\frac{h}{2J}$}\log(N)+\mbox{$\frac{\xi}{2}$}\right)}\\
    &\leq 4(C_2+C_3)N^{\frac{1}{2}}N^{\frac{\lambda h \eta}{J}}e^{-\mbox{$\frac{J-2\lambda\eta}{2}$}\xi}.
\end{align*}
Now, using the fact that $
\mathbbm{1}_{D_{N}\left(t\right)\geq \sqrt{N}}\leq D_{N}(t)/\sqrt{N}$ and the previous display,
we have
\begin{align*}
    \mathbb{P}_{w_0,z_0}&\left(D_{N}\left(\mbox{$\frac{h}{2J}$}\log(N)+\mbox{$\frac{\xi}{2}$}\right)\geq \sqrt{N}, \, \tau_{\rm{exit}}>\mbox{$\frac{h}{2J}$}\log(N)+\mbox{$\frac{\xi}{2}$}\right) \\
    	 &\leq N^{-1/2}
    \mathbb{E}_{w_0, z_0}\left(D_N\left(\mbox{$\frac{h}{2J}$}\log(N)+\mbox{$\frac{\xi}{2}$}\right)
    \mathbbm{1}_{\tau_{\rm{exit}}>\mbox{$\frac{h}{2J}$}\log(N)+\mbox{$\frac{\xi}{2}$}}\right)\\
    &\leq 4(C_2+C_3)N^{\frac{\lambda h \eta}{J}}e^{-\mbox{$\frac{J-2\lambda\eta}{2}$}\xi}.
\end{align*}
The proof is complete after noting that, as $\eta(N)=O(N^{-\frac{1-h}{2}})$, we have $\lim_{N\to\infty} N^{\frac{\lambda h}{J}\eta}=1$.
\end{proof}
\subsection{Final phase}\label{sec:finphase}
We now show that, once the two copies are within distance $\sqrt{N}$ away from each other, the additional time to coalescence is at most $\xi/2>0$ (which is an absolute constant), with high probability, on the event that the copies stay in the good set $\widehat{S}$ for a time that is at least $\xi/2$. 
\begin{lemma}\label{final}
Let $\tau_0:=\inf\{t\geq0, D_N(t)=0\}$. Suppose $w_0/N, z_0/N\in\widehat{S}$ and $d_0=z_0-w_0 \le\sqrt{N}$, then for all $N$ sufficiently large such that $J/2\lambda \ge \eta(N)$, the following relation holds for all $\xi>0$,
\begin{equation*}
    \mathbb{P}_{w_0, z_0}\left(\tau_0>\mbox{$\frac{\xi}{2}$},\tau_{\rm{exit}}>\mbox{$\frac{\xi}{2}$}\right)\leq\frac{4}{\sqrt{\mu\wedge\varepsilon}}\xi^{-\frac{1}{2}}.
\end{equation*}
\end{lemma}
To prove this lemma, we require \cite[Proposition 4.1]{barbour2022long} (stated below as Proposition~\ref{RWBound}), which is a continuous time analogue of \cite[Proposition 17.19]{levin2017markov}, and controls the tails of certain hitting times.
With the right setup, the proof of Lemma~\ref{final} is a straightforward application of the lemma.
\begin{proposition}\cite[Proposition 4.1]{barbour2022long}\label{RWBound}
Let $X(t)$ be a continuous-time Markov jump chain with state space $S$ and rate matrix $Q$, which is stable, conservative and non-explosive. Let $B$ and~$\sigma^2$ be positive, and let $f: S \to \mathbb{R}_+$ be a function. Set $S_0 := \{x\in S : f (x) = 0\}$, and assume that
\begin{enumerate}
    \item the drift $\sum_y Q(x, y)(f(y)-f(x))$ of $f$ is non-positive for all $x \in S \backslash S_0$;
    \item $f(X)$ makes jumps of magnitude at most $B$;
    \item $\sum_y Q(x, y)(f(y)-f(x))^2\geq \sigma^2$ for all $x \in S \backslash S_0$.
\end{enumerate}
Define $T_{*}:=\inf\{t:f(X(t))=0\}$ the hitting time of $S_0$. Then, for any $t_0\geq 2B^2/\sigma^2$, 
\begin{equation*}
    \mathbb{P}(T_{*}\geq t_0)\leq \frac{2\sqrt{2}f(X_0)}{\sigma\sqrt{t_0}}.
\end{equation*}
\end{proposition}
The drift condition of the proposition does not hold globally for $X_N$, but it does hold in the good set $\widehat{S}$. Since we are only interested in the
behaviour of the chain before it leaves the good set,  conceptually the lemma still applies. To apply it directly,
we introduced a modified version of the chain that  reflects at the ``boundary'' of the good set. Let $\ell:=\left\{\left\lceil (x^{\star}-\eta)N \right\rceil-1\right\}$ and $u:=\left\{\left\lfloor (x^{\star}+\eta)N \right\rfloor+1\right\}$, be the states corresponding to leaving the good set $\widehat S$. Define a state space $S:=N\widehat{S}\cup u \cup\ell$. Let $\overline{X}_N(t)$ be a Markov chain on $S$, which evolves as follows. If $\overline{X}_N(t)/N\in\widehat{S}$,
then it makes jumps according to the transition rates of the original process $X_N(t)$. Otherwise, we have 
\begin{equation}\label{2211}
\begin{split}
\openup\jot 
\begin{aligned}[t]
    & u \to u-1 \text{\quad at rate \quad} \mathcal{C}_{u},\\
    & u \to u+1 \text{\quad at rate \quad} 0,
\end{aligned}
\qquad\qquad 
\begin{aligned}[t]
    & \ell \to \ell-1 \text{\quad at rate \quad} 0,\\
    & \ell \to \ell+1 \text{\quad at rate \quad} \mathcal{I}_{\ell}.
\end{aligned}
\end{split}
\end{equation}
The processes $X_N(t)$ and $\overline{X}_N(t)$ admit a natural coupling, under which the two processes start from the same initial state $x_0/N\in\widehat{S}$ and evolve together until time $\tau^{X}_{\rm{exit}}=\inf\{t\geq 0: X_N(t)/N\notin\widehat{S}\}$ defined in \eqref{xexit}. This leads to the following immediate observation: for any event $A$ in the stopped sigma-algebra of $\tau_{\rm{exit}}^X$, one has $\mathbb{P}(A)=\overline{\mathbb{P}}(A)$, where we use $\mathbb{P}$ to denote the underlying probability measure of $X_N(t)$ and $\overline{\mathbb{P}}$ to denote the underlying probability measure of $\overline{X}_N(t)$. That is, the two processes $X_N(t)$ and $\overline{X}_N(t)$ are path-wise indistinguishable in probability up to time $\tau^{X}_{\rm{exit}}$. 

As before, let $(W_N(t), Z_N(t))$ be two independent copies of $X_N(t)$ with transitions specified by \eqref{couplingrule} before coalescence, and moving together after coalescence. Also, let $(\overline{W}_N(t), \overline{Z}_N(t))$ evolve as two independent copies of $\overline{X}_N(t)$ with $\overline{W}_N(0)\leq\overline{Z}_N(0)$ until coalescence, and moving together after coalescence. Analogously to the relationship between $X_N(t)$ and $\overline{X}_N(t)$, the two-dimensional processes $(W_N(t), Z_N(t))$ and $(\overline{W}_N(t), \overline{Z}_N(t))$ are path-wise indistinguishable in probability up to the first time $(W_N(t), Z_N(t))$ exits $N\widehat{S}\times N\widehat{S}$, that is until  the time $\tau_{\rm{exit}}$ defined in \eqref{texit}.    
\begin{proof}[Proof of Lemma~\ref{final}]
Set
\begin{align*}
    \tau^{\overline{W}, \overline{Z}}_{0}&:=\inf\{t\geq 0: \overline{Z}_N(t)-\overline{W}_N(t)=0\}, \\
     \tau^{\overline{W}, \overline{Z}}_{\rm{exit}}&:=\inf\{t\geq 0: \overline{W}_N(t)=\ell \text{ or } \overline{Z}_N(t)=u\},
\end{align*}
and write $\mathbb{P}$ for the underlying probability measure of $(W_N(t), Z_N(t))$ and $\overline{\mathbb{P}}$ for the underlying probability measure of $(\overline{W}_N(t), \overline{Z}_N(t))$. Since  $\overline{W}_N(t)\leq \overline{Z}_N(t)$ for all $t\geq0$, the process $(\overline{W}_N(t), \overline{Z}_N(t))$ leaves $N\widehat{S}\times N\widehat{S}$ either by having the lower one jump down or by having the upper one jump up. Then, as in the  previous discussion of the relationship between $X_N(t)$ and $\overline{X}_N(t)$, we have the following two immediate observations: 
\begin{align}
    \mathbb{P}_{w_0, z_0}\left(\tau_{\rm{exit}}\leq \mbox{$\frac{\xi}{2}$}\right)
    	&=\overline{\mathbb{P}}_{w_0, z_0}\left(\tau^{\overline{W}, \overline{Z}}_{\rm{exit}}\leq \mbox{$\frac{\xi}{2}$}\right); \label{ob1}\\
    \mathbb{P}_{w_0, z_0}\left(\tau_0\leq\mbox{$\frac{\xi}{2}$},\,\tau_{\rm{exit}}> \mbox{$\frac{\xi}{2}$}\right)
    	&=\overline{\mathbb{P}}_{w_0, z_0}\left(\tau^{\overline{W}, \overline{Z}}_{0}\leq\mbox{$\frac{\xi}{2}$},\, \tau^{\overline{W}, \overline{Z}}_{\rm{exit}}> \mbox{$\frac{\xi}{2}$}\right). \label{ob2}
\end{align}

Therefore, it suffices to prove the result for the process $(\overline{W}_N(t), \overline{Z}_N(t))$. 
To  apply \cite[Proposition 4.1]{barbour2022long} (Proposition \ref{RWBound}) to $(\overline{W}_N(t), \overline{Z}_N(t))$, first set
\begin{equation*}
 f((w,z))=(z-w)\mathbbm{1}_{w\leq z}.
\end{equation*}
Clearly, $f$ is non-negative and integer-valued, so $f(\textbf{x})\geq 1$ for all $\textbf{x}=(w,z)$ such that $f(\textbf{x})\neq 0$. 
Also, Condition (2) of Proposition~\ref{RWBound}
 is satisfied by taking $B=1$. 

We now verify Condition (1) of Propositions~\ref{RWBound}, i.e. the drift, $\sum_{\textbf{y}} Q(\textbf{x}, \textbf{y})(f(\textbf{y})-f(\textbf{x}))$ is non-positive for all $\textbf{x}$ such that $f(\textbf{x})\neq 0$. First note that, in our case, there are four terms in the sum since, before coalescence, the coupled process $(\overline{W}_N(t), \overline{Z}_N(t))$ has four possible ways of jumping. The movements 
\begin{align*}
    &\textbf{x}=(\bar{w},\bar{z})\mapsto \textbf{y}=(\bar{w}+1,\bar{z})\text{\quad with \quad} Q(\textbf{x}, \textbf{y})=\mathcal{I}_{\bar{w}},\\
    &\textbf{x}=(\bar{w},\bar{z})\mapsto \textbf{y}=(\bar{w},\bar{z}-1)\text{\quad with \quad} Q(\textbf{x}, \textbf{y})=\mathcal{C}_{\bar{z}},
\end{align*}
are always allowed and lead to $f(\textbf{y}) - f(\textbf{x}) = -1$. The other two possible jumps, 
\begin{align*}
    &\textbf{x}=(\bar{w},\bar{z})\mapsto \textbf{y}=(\bar{w},\bar{z}+1)\text{\quad with \quad} Q(\textbf{x}, \textbf{y})=\mathcal{I}_{\bar{z}}\cdot\mathbbm{1}\{\bar{z}\neq u\},\\
    &\textbf{x}=(\bar{w},\bar{z})\mapsto \textbf{y}=(\bar{w}-1,\bar{z})\text{\quad with \quad} Q(\textbf{x}, \textbf{y})=\mathcal{C}_{\bar{w}}\cdot\mathbbm{1}\{\bar{w}\neq \ell\},
\end{align*}
always leads to $f(\textbf{y}) - f(\textbf{x}) = 1$ and the boundary conditions are reflected by the indicators in the transition rates. Also, note that for all $x\in S$,
    \begin{equation*}
        \lvert x/N-x^{\star} \rvert \leq \eta+N^{-1}=O(N^{-\frac{1-h}{2}}).
    \end{equation*}
with $\eta\equiv\eta(N)$ defined in \eqref{goodset}. It then follows that
\begin{align}\label{contracting_coupling}
    \smash{\sum_{\textbf{y}}} Q(\textbf{x}, \textbf{y})(f(\textbf{y})-f(\textbf{x}))
    &= (\mathcal{I}_{\bar{z}}\cdot\mathbbm{1}\{\bar{z}\neq u\}+\mathcal{C}_{\bar{w}}\cdot\mathbbm{1}\{\bar{w}\neq \ell\})-(\mathcal{I}_{\bar{w}}+\mathcal{C}_{\bar{z}})\nonumber\\
    &\leq(\mathcal{I}_{\bar{z}}-\mathcal{I}_{\bar{w}})-(\mathcal{C}_{\bar{z}}-\mathcal{C}_{\bar{w}})\nonumber\\
    & =(\lambda-\mu-\varepsilon) (\bar z-\bar w) - \mbox{$\frac{\lambda}{N}$}(\bar z-\bar w)(\bar z+\bar w)\nonumber\\
    &\leq \left((\lambda-\mu-\varepsilon) - 2\lambda\left(x^{\star}-\eta-N^{-1}\right)\right)(\bar z-\bar w) \nonumber\\
    & = -(J-2\lambda\eta-2\lambda N^{-1})\lvert \bar z- \bar w\rvert,
\end{align}

which is clearly negative, when $N$ is sufficiently large, for $\bar z\not= \bar w$. To get the $\sigma^2$ satisfying Condition (3), note that  
\begin{align*}
    \sum_{\textbf{y}} Q(\textbf{x}, \textbf{y})(f(\textbf{y})-f(\textbf{x}))^2 &= \mathcal{I}_{\bar{w}} + \mathcal{I}_{\bar{z}}\cdot\mathbbm{1}\{\bar{z}\neq u\} + \mathcal{C}_{\bar{z}} + \mathcal{C}_{\bar{w}}\cdot\mathbbm{1}\{\bar{w}\neq \ell\}\\
    &\geq \mathcal{I}_{\bar{w}} + \mathcal{C}_{\bar{z}} \geq \mathcal{I}_{\bar{w}} + \mathcal{C}_{\bar{w}} \geq(\mu \wedge \varepsilon)N:=\sigma^2,
\end{align*}
since the minimum of 
\begin{equation*}
    g(x):=\mathcal{I}_{x}+\mathcal{C}_{x}=-\frac{\lambda}{N} x^2+(\lambda+\mu-\varepsilon)x+\varepsilon N
\end{equation*}
is achieved either at $g(0)=\varepsilon N$ or $g(N)=\mu N$. An application of Proposition \ref{RWBound} gives
\begin{equation}\label{resrwb}
    \overline{\mathbb{P}}_{w_0,z_0}\left(\tau^{\overline{W}, \overline{Z}}_{0}\wedge \tau^{\overline{W}, \overline{Z}}_{\rm{exit}}>\mbox{$\frac{\xi}{2}$}\right)\leq \overline{\mathbb{P}}_{w_0,z_0}\left(\tau^{\overline{W}, \overline{Z}}_{0}>\mbox{$\frac{\xi}{2}$}\right)\leq \frac{2\sqrt{2}d_0}{\sqrt{(\mu \wedge \varepsilon)N}\sqrt{\xi/2}}\leq \frac{4}{\sqrt{\mu \wedge \varepsilon}}\xi^{-\frac{1}{2}}.
\end{equation}
Then in view of observation \eqref{ob1} and \eqref{ob2}, we conclude that 
\begin{align*}
     \mathbb{P}_{w_0, z_0}\left(\tau_0\leq\xi/2,\tau_{\rm{exit}}>\xi/2\right)&=\overline{\mathbb{P}}_{w_0, z_0}\left(\tau^{\overline{W}, \overline{Z}}_{0}\leq\xi/2,\, \tau^{\overline{W}, \overline{Z}}_{\rm{exit}}>\xi/2\right)\\
    &= \overline{\mathbb{P}}_{w_0,z_0}\left(\tau^{\overline{W}, \overline{Z}}_{\rm{exit}}>\xi/2\right)-\overline{\mathbb{P}}_{w_0,z_0}\left(\tau^{\overline{W}, \overline{Z}}_{0}\wedge \tau^{\overline{W}, \overline{Z}}_{\rm{exit}}>\xi/2\right)\\
    &\geq\mathbb{P}_{w_0,z_0}\left(\tau_{\rm{exit}}>\xi/2\right)-\frac{4}{\sqrt{\mu \wedge \varepsilon}}\xi^{-\frac{1}{2}},
\end{align*}
and hence
\begin{align*}
    \mathbb{P}_{w_0, z_0}\left(\tau_0>\xi/2,\tau_{\rm{exit}}>\xi/2\right)&=\mathbb{P}_{w_0,z_0}\left(\tau_{\rm{exit}}>\xi/2\right)-\mathbb{P}_{w_0, z_0}\left(\tau_0\leq\xi/2,\tau_{\rm{exit}}>\xi/2\right)\\
    &\leq\frac{4}{\sqrt{\mu \wedge \varepsilon}}\xi^{-\frac{1}{2}},
\end{align*}
as desired.
\end{proof}
\subsection{Bounding the time to coalescence.}
We combine the previous results from this section for a proof of Lemma \ref{rapidmixing}.   
\begin{proof}[Proof of Lemma \ref{rapidmixing}]
Fix $\xi>0$. Denote the good event that the coupled processes $W_N(t)$ and $Z_N(t)$ are in the interior good set $\widetilde S$ by time $\frac{1-h}{2J}\log(N)$ by
\[
B_N:=\left\{N^{-1}
	W_N\left(\mbox{$\frac{1-h}{2J}$}\log(N)\right)\in\widetilde{S}, N^{-1}Z_N\left(\mbox{$\frac{1-h}{2J}$}\log(N)\right)\in\widetilde{S}\right\}.
	\]
Then the first part of Corollary~\ref{burn} implies
\begin{equation}
\begin{split}\label{eq:t12}
\mathbb{P}_{w_0,z_0}&\left(\tau_{\rm{couple}}>\mbox{$\frac{1-h}{2J}$}\log(N)+\mbox{$\frac{h}{2J}$}\log(N)+\xi\right) \\
	&\leq \mathbb{P}_{w_0,z_0}\left(\tau_{\rm{couple}}>\mbox{$\frac{1-h}{2J}$}\log(N)+\mbox{$\frac{h}{2J}$}\log(N)+\xi, B_N \right)
	+ 8e^{-C_1N^h}
	\end{split}
\end{equation}
By the Markov property, to bound the first term on the right hand side of~\eqref{eq:t12}, 
it is enough to bound, uniformly in $w_0/N, z_0/N\in\widetilde{S}\subset\widehat{S}$,
\begin{align*}
\mathbb{P}_{w_0,z_0}&\left(\tau_{\rm{couple}}>\mbox{$\frac{h}{2J}$}\log(N)+\xi\right) \\
 	&\leq \mathbb{P}_{w_0,z_0}\left(\tau_{\rm{couple}}>\mbox{$\frac{h}{2J}$}\log(N)+\xi, \tau_{\rm{exit}}>t_{\rm{follow}}  \right)+ 8e^{-C_1N^h} \\
	&\leq \mathbb{P}_{w_0,z_0}\left(\tau_{\rm{couple}}>\mbox{$\frac{h}{2J}$}\log(N)+\xi, \tau_{\rm{exit}}>\mbox{$\frac{h}{2J}$}\log(N)+\xi  \right)+ 8e^{-C_1N^h}
\end{align*}
where the first inequality follows from the second part of Corollary~\ref{burn}, and for the second inequality we assume
$N$ is sufficiently large so that $\mbox{$\frac{h}{2J}$}\log(N)+\xi\leq t_{\rm{follow}}$.

Lemma~\ref{med} and Lemma~\ref{final} combined with a straightforward application of the Markov property implies that for all $N$ large enough,
\begin{equation*}
\mathbb{P}_{w_0,z_0}\left(\tau_{\rm{couple}}>\mbox{$\frac{h}{2J}$}\log(N)+\xi,\, \tau_{\rm{exit}}>\mbox{$\frac{h}{2J}$}\log(N)+\xi\right)\leq \frac{4}{\sqrt{\mu \wedge \varepsilon}}\xi^{-\frac{1}{2}}+C_4e^{-\mbox{$\frac{J}{2}$}\xi}.
\end{equation*}

Combining the displays above gives 
\begin{equation*}
\mathbb{P}_{w_0,z_0}\left(\tau_{\rm{couple}}>\mbox{$\frac{1-h}{2J}$}\log(N)+\mbox{$\frac{h}{2J}$}\log(N)+\xi\right)\leq \frac{4}{\sqrt{\mu \wedge \varepsilon}}\xi^{-\frac{1}{2}}+C_4e^{-\mbox{$\frac{J}{2}$}\xi}+16e^{-C_1N^{h}},
\end{equation*}
and taking the limit superior as $N \to \infty$ on both sides finishes the proof.
\end{proof}

\section{Proof of the lower bound for the mixing time}\label{lowerSIS}
In this section, we prove the following lower bound on the mixing time, which together with Lemma \ref{rapidlymixing} proves Theorem \ref{cutoff}. Recall $\rho_N(t)=\max_{x\in\{0,1,\ldots,N\}}\lVert P^t_N(x,\cdot)-\pi_N \rVert_{TV}$ and $t_N=\frac{1}{2J}\log(N)$ from Definition~\ref{cutoffphenomenon} and Theorem~\ref{cutoff}.
\begin{lemma}\label{lowerbound}
For sufficiently large $\xi\geq0$, we have 
\begin{equation*}
    \liminf_{N\to\infty}\rho_N\left(t_N-\xi\right)\geq 1-\psi_2(\xi),
\end{equation*}
where $\psi_2(\cdot)$ is a non-negative real valued function depending only on $\lambda, \mu$ and $\varepsilon$ with $\psi_2(\xi)=O(1/\sqrt{\xi})$ as $\xi\to\infty$. 
\end{lemma}
Before outlining the strategy of the proof, we first state an improved version of the concentration of measure inequality from Lemma~\ref{following} for the scaled process $X_N(t)/N$, which holds only if the starting state of the chain $X_N(t)$ is close enough to $x^{\star}N$. Recall the interval of the form $I(r)=[x^{\star}-r, x^{\star}+r]$ defined in \eqref{def:interval}. In the previous section, we took $r=\eta(N)$, in defining our good set $\widehat{S}$, so that the size of the good set shrinks as $N$ gets large. To prove the lower bound, we consider $r<J/6\lambda$ as a positive constant, i.e.\ independent of $N$.
\begin{lemma}\label{const.follow}
Fix $h\in(0,1)$, $r < J/6\lambda$ and $\xi>0$. Let $x_0/N\in I(\frac{r}{2})$ and let $t\leq t_{\rm{follow}}=\mbox{$\frac{1}{J}$}\lceil e^{C_1N^{h}}\rceil$. Then the following holds for all $N$ large enough (depending on $h, \xi, \lambda, \mu$ and $\varepsilon$):
\begin{equation*}
    \mathbb{P}_{x_0}\left(\left\lvert X_N(t)-\mathbb{E}_{x_0}(X_N(t))\right\rvert \geq 2\xi\sqrt{N}\right)\leq 3e^{-\frac{J\xi^2}{3(\lambda+\mu+\varepsilon)}}.
\end{equation*}
\end{lemma}
The lemma essentially follows from an application of \cite[Theorem 3.3]{barbour2022long}, but the details are technical, so we postpone the proof until the end of this section. Assuming that Lemma~\ref{const.follow} holds, the proof of Lemma \ref{lowerbound} is then broken down into two steps. 

The first step is taken in Lemma \ref{concentration.pi}, where we show that the stationary distribution $\pi_N$ is 
with high probability concentrated in a ball of order $\sqrt{N}$ around the fixed point $x^{\star} N$. The explicit expression for the stationary distribution can be obtained by solving the detailed balance equation, (see for example the displays (A1) and (A2) in \cite{mieghem_2020}), but it is not easy to observe the concentration from this expression. The outline of our strategy is to first fix $\xi>0$ and choose $N$ large enough such that $t_N+\xi<t_{\rm{follow}}$. We use the improved concentration inequality of Lemma~\ref{const.follow} to show that the distribution of $X_N(t_N+\xi)$ is concentrated in a $2\xi\sqrt{N}$-ball around $\mathbb{E}_{x_0}(X_N(t_N+\xi))$ given $x_0/N\in I(\frac{r}{2})$. Next, the first part of Lemma~\ref{mean} implies that $\mathbb{E}_{x_0}(X_N(t_N+\xi))$ is within $e^{-J\xi}K_1\sqrt{N}$ of $x^{\star}N$. Together these imply that $X_N(t_N+\xi)$ is in a $(e^{-J\xi}K_1+2\xi)\sqrt{N}$ ball around $x^{\star}N$ with high probability. Finally, since the distribution of $X_N$ is close to stationary by the time $t_N+\xi$, as stated in Lemma~\ref{rapidlymixing}, this implies that $\pi_N$ also assigns most of the mass to a $(e^{-J\xi}K_1+2\xi)\sqrt{N}$ ball around $x^{\star}N$.

To prove the lower bound, it is sufficient to only consider one specific initial state. For convenience, we take the initial state to be $\bar{x}:=\lfloor (x^{\star}+\frac{r}{2})N \rfloor$. The second step of the proof of Lemma~\ref{lowerbound} is then to show that the distribution of $X_N(t_N-\xi)$ given $X_N(0)=\bar{x}$ is concentrated in a region that is sufficiently apart from $x^{\star}N$; combined with the concentration of $\pi_N$ around $x^{\star}N$, this gives a lower bound on the mixing time. The outline of the proof of this result is as follows. First fix a $\xi$ large enough such that $e^{J\xi}K_2-2\xi>e^{-J\xi}K_1+2\xi$, where $K_1, K_2$ are as in Lemma~\ref{mean}. By the improved concentration inequality, the distribution of $X_{N}(t_N-\xi)$ is concentrated in a $2\xi\sqrt{N}$-ball around $\mathbb{E}_{\bar{x}}(X_{N}(t_N-\xi))$. The second part of Lemma~\ref{mean} then implies that $\mathbb{E}_{\bar{x}}(X_{N}(t_N-\xi))$ at least $e^{J\xi}K_2\sqrt{N}$ above the fixed point $x^{\star}N$. Together, these two results indicate that with high probability 
\[
    X_{\bar{x}}(t_N-\xi)- x^{\star}N \geq (e^{J\xi}K_2-2\xi)\sqrt{N}>(e^{-J\xi}K_1+2\xi)\sqrt{N},
\] 
but by Lemma \ref{concentration.pi}, we know that $\pi_N$ assigns most of its mass within $(e^{-J\xi}K_1+2\xi)\sqrt{N}$ of   $x^{\star}N$.

Now, with the outline in mind, we state and prove concentration for the stationary distribution $\pi_N$.
\begin{lemma}\label{concentration.pi}
For any $\xi>0$, define the set
\begin{equation*}
    \mathcal{S}_N:=\left\{x\in\{0,1,\ldots,N\}:\left\lvert x-x^{\star}N\right\rvert \leq (e^{-J\xi}K_1+2\xi)\sqrt{N}\right\},
\end{equation*}    
where $K_1$ is the constant defined in Lemma \ref{mean}. One has 
\begin{equation*}
    \limsup_{N\to\infty}\pi_N(\mathcal{S}_N^c)\leq \psi_1(\xi)+3e^{-\frac{J\xi^2}{3(\lambda+\mu+\varepsilon)}},
\end{equation*}
where $\psi_1(\xi)=O(1/\sqrt{\xi})$ is as defined in Lemma \ref{rapidlymixing}.
\end{lemma}
\begin{proof}
Let $x_0/N\in I(\frac{r}{2})$, let $h\in(0,1)$ and let
$\xi>0$. 
Lemma \ref{const.follow} implies that, as long as $N$ is sufficiently large, 
\begin{align*}
    \mathbb{P}_{x_0}\left(\left\lvert X_N\left(t_N+\xi
    \right)-\mathbb{E}_{x_0}\left(X_N(t_N+\xi)\right)\right\rvert\leq 2\xi\sqrt{N}\right)\geq1-3e^{-\frac{J\xi^2}{3(\lambda+\mu+\varepsilon)}}.
\end{align*}
Moreover, applying the first part of Lemma \ref{mean} with $c=\xi$, we obtain 
\begin{equation*}
    \lvert \mathbb{E}_{x_0}(X_N(t_N+\xi))-x^{\star}N \rvert \leq e^{-J\xi}K_1\sqrt{N},
\end{equation*}
for $N$ large enough. It then follows that, for $N$ large enough,
\begin{align*}
    \mathbb{P}_{x_0}&\left(\left\lvert X_N(t_N+\xi)-x^{\star}N\right\rvert \leq (2\xi+e^{-J\xi}K_1)\sqrt{N}\right) \\
    &\geq \mathbb{P}\left(\left\lvert X_N(t_N+\xi)-\mathbb{E}_{x_0}\left(X_N(t_N+\xi)\right)\, N\right\rvert \leq 2\xi\sqrt{N}\right)\\
    &\geq 1-3e^{-\frac{J\xi^2}{3(\lambda+\mu+\varepsilon)}}.
\end{align*}
Thus, for any $\xi>0$ and $N$ large enough, the inequality implies that
\begin{align*}
    \pi_N(\mathcal{S}_N^c)& \leq P_N^{t_N+\xi}(x_0,\,\mathcal{S}_N^c)+\rho_N(t_N+\xi) \\
    	&\leq 3e^{-\frac{J\xi^2}{3(\lambda+\mu+\varepsilon)}} + \rho_N(t_N+\xi),
\end{align*}
where $P_N^{t}(x_0,\,\cdot)$ is the law of $X_N(t)$ given the initial state $x_0$. The result now follows by taking the limsup and applying  Lemma \ref{rapidlymixing}.
\end{proof}

We can now prove our lower bound on the mixing time.

\begin{proof}[Proof of Lemma~\ref{lowerbound}]
Fix a large enough $\xi>0$ such that $e^{J\xi}K_2-2\xi>e^{-J\xi}K_1+2\xi$ and let $\bar{x}=\lfloor (x^{\star}+\frac{r}{2})N \rfloor$. It follows from the second part of Lemma \ref{mean} by taking $c=-\xi$ that  
\begin{equation*}
    \mathbb{E}_{\bar{x}}(X_N(t_N-\xi))-x^{\star}N \geq e^{J\xi}K_2\sqrt{N},
\end{equation*}
for all $N$ large enough. Moreover, note that $\bar{x}/N\in I(\frac{r}{2})$ and so for any $N$ such that $t_N-\xi\leq t_{\rm{follow}}$, Lemma \ref{const.follow} implies that
\begin{align*}
    \mathbb{P}_{\bar{x}}\left(\left\lvert X_N(t_N-\xi)-
    \mathbb{E}_{\bar{x}}\left(X_N(t_N-\xi)\right)\right\rvert \leq 2\xi\sqrt{N}\right)\geq1-3e^{-\frac{J\xi^2}{3(\lambda+\mu+\varepsilon)}}.
\end{align*}
for all $N$ large enough. Combining the previous two relations gives 
\begin{equation}\label{P_N}
    P_N^{t_N-\xi}(\bar{x}, \mathcal{S}_N)\leq3e^{-\frac{J\xi^2}{3(\lambda+\mu+\varepsilon)}},
\end{equation}
where $\mathcal{S}_N$ is the set defined in Lemma \ref{concentration.pi}. Now, 
\begin{align*}
    \rho_N(t_N-\xi) &\geq \lVert P_N^{t_N-\xi}(\bar{x}, \cdot)-\pi_N \rVert_{TV} \geq \lvert \pi_N(\mathcal{S}_N)-P_N^{t_n-\xi}(\bar{x}, \mathcal{S}_N)\rvert,
\end{align*}
so that \eqref{P_N} and Lemma \ref{concentration.pi} imply that for all sufficiently large $\xi$,
\begin{equation*}
    \liminf_{N\to\infty} \rho_N(t_N-\xi)\geq 1-\psi_1(\xi)-6e^{-\frac{J\xi^2}{3(\lambda+\mu+\varepsilon)}}.
\qedhere\end{equation*}
\end{proof}

The only thing left at this point is to prove Lemma \ref{const.follow}. To do this, we shall make use of \cite[Theorem 3.3]{barbour2022long} which provides a concentration inequality for contracting Markov chains.
\begin{theorem}\cite[Theorem 3.3]{barbour2022long}\label{Wconcentration}
    Let $X(t)$ be a stable, conservative, non-explosive continuous-time Markov chain on a discrete state space $E$, with generator matrix $Q:=(Q(x,y): x,y\in E)$. Suppose that $d(\cdot,\,\cdot)$ is a metric on $E$, and let $f : E \to \mathbb{R}$ be a function such that, for some constant $L$, $\lvert f(x)-f(y)\rvert \leq Ld(x, y)$ for all $x, y \in E$. Let $\widehat{E}$ be a subset of $E$, and let $q$ and $D$ be constants such that $-Q(x, x) \leq q$ for all $x \in \widehat{E}$ and $d(x,\,y) \leq  D$ whenever $x \in \widehat{E}$ and $y$ is such that $Q(x, y) > 0$. For $t > 0$, let $A_t = \{X(s) \in \widehat{E} \text{ for } 0 \leq s < t\}$. Suppose there is a Markovian coupling of two copies of $X(t)$ with generator $\mathcal{A}$ and constant $\ell>0$, such that for all $x, y\in E$, 
    \begin{equation}\label{contracting}
        \mathcal{A}d(x,y)\leq -\ell d(x,y).  
    \end{equation}
    Then, for all $x \in \widehat{E}, t > 0$ and $\xi\geq0$,
    \begin{equation*}
        \mathbb{P}_x\left(\left\{\left\lvert f(X(t))-\mathbb{E}_x(f(X(t)))\right\rvert \geq \xi \right\}\cap A_t\right)\leq 2\exp\left(-\frac{\xi^2}{qL^2D^2/\ell+2LD\xi/3}\right).
    \end{equation*}
\end{theorem}

We are not able to apply this theorem directly to our chain $X_N(t)$ since it is not contracting on the entire state space. One could modify the theorem to get a version that requires the contraction property to hold only on a ``good'' subset of the state space, as
is the case for $X_N(t)$ near $x^{\star}N$. (Indeed, \cite[Theorem~2.3]{barbour2022long} is stated in such a form, but for discrete-time.) Instead, we avoid this issue by making use of the reflected chain $\overline{X}_N(t)$ defined for the final phase immediately after Proposition~\ref{RWBound}, which behaves as $X_N(t)$ in a good set and reflects at the boundary. Here, with a bit of abuse of notation, we define $\overline{X}_N(t)$ to be the reflected chain on the state space $E:=NI(r)\cup \{\lfloor (x^{\star}+r)N \rfloor + 1\} \cup \{\lceil (x^{\star}-r)N \rceil - 1\}$, where $\{\lfloor (x^{\star}+r)N \rfloor + 1\}$ and $\{\lceil (x^{\star}-r)N \rceil - 1\}$ are considered to the the upper and lower boundary respectively and $r<J/6\lambda$ is a fixed positive constant.
\begin{proof}[Proof of Lemma \ref{const.follow}]
Recall the definition of $\tau^{X}_{\rm{exit}}(I(r))$ from \eqref{xexit}. Just as we argued in the passage following~\eqref{2211}, it is sufficient to prove the concentration for $\overline{X}_N(t)$ on $A_t=\{\tau^{\overline{X}}_{\rm{exit}}(I(r))\geq t\}$, where $\tau^{\overline{X}}_{\rm{exit}}(I(r)):=\inf\{t\geq0, \overline{X}_N(t)\in \{\lfloor (x^{\star}+r)N \rfloor + 1\}\cup \{\lceil (x^{\star}-r)N \rceil - 1\}\}$. To apply Theorem \ref{Wconcentration}, we take $f$ to be the identity function and the metric $d(x,\,y)=\lvert x-y \rvert$ so that $L=1$. The jump size of $\overline{X}_N(t)$ is clearly bounded by $D=1$ and we have for all $x\in \{0,\ldots,N\}$
    \begin{equation*}
        -Q(x,x) \le \lambda x(1-x/N)+\varepsilon(N-x)+\mu x\leq(\lambda+\mu+\varepsilon)N:=q.
    \end{equation*}
Finally, we show that there is a coupling $(\overline{W}_N(t), \overline{Z}_N(t))$ of two copies of $\overline X_N(t)$ which satisfies \eqref{contracting} for all $(\overline{W}_N(0), \overline{Z}_N(0))=(w,z)\in E\times E$, and hence $\overline{X}_N(t)$ is contracting in Wasserstein distance. The coupling is that used earlier in Section~\ref{sec:finphase}, where the copies move independently until they collide, and then they move in unison. Verifying~\eqref{contracting} for this coupling follows very similarly to verifying Condition (1) in Proposition \ref{RWBound}. 

To compute the left-hand side of \eqref{contracting}, first note that if the two copies have already collided with each other, then the left-hand side is $0$. Otherwise, the left-hand side is the same as in~\eqref{contracting_coupling} with $\eta$ replaced by $r$. Since we have chosen $r$ to be less than $J/6\lambda$, so for sufficiently large $N$ we must have $J-2\lambda r-2\lambda N^{-1}\geq J/2$. Therefore,~\eqref{contracting} is satisfied with $\ell=J/2$. Then, Theorem \ref{Wconcentration} implies that for all $x_0\in \widehat{E}:=NI(r)$, $t\geq0$, and $\xi>0$,
    \begin{align*}              
    \overline{\mathbb{P}}_{x_0}\left(\left\{\left\lvert \overline{X}_N(t)-\overline{\mathbb{E}}_{x_0}(\overline{X}_N(t))\right\rvert \geq \xi\sqrt{N}\right\}\cap \{\tau^{\overline{X}}_{\rm{exit}}(I(r))\geq t\}\right)\leq 2e^{-\frac{\xi^2}{2(\lambda+\mu+\varepsilon)/J+2\xi/(3\sqrt{N})}}.
    \end{align*}
    Moreover, since $X_N(t)$ and $\overline{X}_N(t)$ are path-wise indistinguishable in probability before leaving $NI(r)$, we have 
       \begin{align*}
           \big\lvert \mathbb{E}_{x_0}(X_N(t))-\overline{\mathbb{E}}_{x_0}(\overline{X}_N(t))\big\rvert
        	&=  \left \lvert \mathbb{E}_{x_0}\left(X_N(t)\mathbbm{1}_{\tau^{X}_{\rm{exit}}(I(r))< t} \right)-\overline{\mathbb{E}}_{x_0}\left(\overline{X}_N(t)\mathbbm{1}_{\tau^{\overline X}_{\rm{exit}}(I(r))< t}\right) \right\rvert\\
        &\leq N \mathbb{P}_{x_0}(\tau^{X}_{\rm{exit}}(I(r))< t).
    \end{align*}    
The second part of Lemma~\ref{burn_x} 
implies that for all $t\leq t_{\rm{follow}}$, $x_0/N\in I(\frac{r}{2})\subset I(r)$ and all sufficiently large $N$,  we have
\begin{equation}\label{eq: exit_Ir_bound}
    \mathbb{P}_{x_0}(\tau^{X}_{\rm{exit}}(I(r))< t)\leq \mathbb{P}_{x_0}(\tau^{X}_{\rm{exit}}(I(r))\leq t_{\rm{follow}})\leq 4e^{-C_1N^h},
\end{equation}
 and hence 
    \begin{equation*}
        \left\lvert \mathbb{E}_{x_0}(X_N(t))-\overline{\mathbb{E}}_{x_0}(\overline{X}_N(t)) \right\rvert=o(1).
    \end{equation*}
    Now, for any fixed $\xi>0$, choose $N$ large enough such that $\left\lvert \mathbb{E}_{x_0}(X_N(t))-\overline{\mathbb{E}}_{x_0}(\overline{X}_N(t)) \right\rvert<\xi\sqrt{N}$. It follows that 
    \begin{align*}
        \mathbb{P}_{x_0}&\left(\left\{\left\lvert X_N(t)-\mathbb{E}_{x_0}(X_N(t))\right\rvert \geq 2\xi\sqrt{N}\right\}\cap \{\tau^{X}_{\rm{exit}}(I(r))\geq t\}\right)\\ 
        &\leq \overline{\mathbb{P}}_{x_0}\left(\left\{\left\lvert \overline{X}_N(t)-\overline{\mathbb{E}}_{x_0}(\overline{X}_N(t))\right\rvert \geq \xi\sqrt{N}\right\}\cap \{\tau^{\overline{X}}_{\rm{exit}}(I(r))\geq t\}\right)\leq 2e^{-\frac{\xi^2}{2(\lambda+\mu+\varepsilon)/J+2\xi/(3\sqrt{N})}}.
    \end{align*}
Together with \eqref{eq: exit_Ir_bound}, we have 
    \begin{align*}
        \mathbb{P}_{x_0}\left(\left\lvert X_N(t)-\mathbb{E}_{x_0}(X_N(t))\right\rvert \geq 2\xi\sqrt{N}\right)
        	&\leq 2e^{-\frac{\xi^2}{2(\lambda+\mu+\varepsilon)/J+2\xi/(3\sqrt{N})}}+4e^{-C_1N^h}\\
        & \leq 3e^{-\frac{J\xi^2}{3(\lambda+\mu+\varepsilon)}},
    \end{align*}
for $N$ large enough, as required.
\end{proof}

\printbibliography

\end{document}